\documentclass[10pt]{amsart}
\usepackage{amsmath,amssymb,verbatim}
\usepackage[utf8]{inputenc}
\usepackage{enumerate}
\usepackage{pdflscape}
\usepackage{amsthm}
\usepackage{mathrsfs}
\usepackage{color}
\usepackage[normalem]{ulem}
\usepackage{cancel}
\usepackage{tikz}
\usetikzlibrary{matrix}
\usepackage[all]{xy}
\usepackage{mathtools}
\usepackage{arydshln}
\usepackage[shortlabels]{enumitem}
 \usepackage[english]{babel}
\usepackage{hyperref}
\hypersetup{
    colorlinks = true,
    linkbordercolor = {white},
    allcolors = blue
} 
\usepackage[capitalize, noabbrev]{cleveref}
 
\makeatletter
\@namedef{subjclassname@1991}{\textup{2020} Mathematics Subject Classification}
\makeatother

\topskip=\baselineskip 
\newtheorem{theorem}{Theorem}[section]
\newtheorem{lemma}[theorem]{Lemma}

\newtheorem{proposition}[theorem]{Proposition}

\newtheorem{remark}[theorem]{Remark}

\newtheorem*{theorem*}{Theorem}

\newtheorem{maintheorem}{Theorem}

\newtheorem{mainproposition}[maintheorem]{Proposition}

\makeatletter
\@addtoreset{equation}{section}
\makeatother

\newcommand{\M}{{\mathrm{D}}}

\newcommand{\GG}{{\Gamma}}

\newcommand{\Cen}{{\rm C}}

\newcommand{\Imagen}{\mbox{\rm Im }}

\newcommand{\inv}{\textup{inv}}

\newcommand{\G}{\mathcal{G}}

\newcommand{\GEN}[1]{\left\langle #1 \right\rangle}

\newcommand{\aug}[1]{\mathrm{I}(#1)}
\newcommand{\augNor}[2]{\mathrm{I}(#1)k#2}
\newcommand{\Aug}[2]{\mathrm{J}^{#1}(#2)}

\newcommand{\ZZ}{\mathrm{Z}}

\newcommand{\Ese}[2]{\mathcal{S}\left(#1\mid #2\right)}

\newcommand{\Ge}{\Gamma}

\DeclareMathOperator{\Cl}{Cl}

\newcommand{\qand}{\quad \text{and} \quad}

\title[On the Modular Isomorphism Problem for  2-generated groups with cyclic derived subgroup]{On the Modular Isomorphism Problem for  \\ 2-generated groups with cyclic derived subgroup}

\author{Diego Garc\'{\i}a-Lucas}
 \address{Departamento de Matem\'aticas, Universidad de Murcia, Spain}
\email{diego.garcial@um.es}

\author{\'{A}ngel del R\'{i}o}
 \address{Departamento de Matem\'aticas, Universidad de Murcia, Spain}
\email{adelrio@um.es}
 
\thanks{Partially supported by Grant PID2020-113206GB-I00 funded by MCIN/AEI/10.13039/501100011033  and by Grant Fundación Séneca 22004/PI/22. }



\date{\today}

\begin{document}

\begin{abstract}
We continue the analysis of the Modular Isomorphism Problem for $2$-generated $p$-groups with cyclic derived subgroup, $p>2$, started in \cite{GLdRS2022}. We show that if $G$ belongs to this class of groups, then the isomorphism type of the quotients $G/(G')^{p^3}$ and $G/\gamma_3(G)^p$ are determined by its modular group algebra. In fact, we obtain a more general but  technical result, expressed in terms of the classification \cite{OsnelDiegoAngel}. We also show that for groups in this class of order at most $p^{11}$, the Modular Isomorphism Problem has positive answer. Finally, we describe  some families of groups of order $p^{12}$  whose group  algebras over the field with $p$ elements cannot be distinguished with the  techniques available to us.
\end{abstract}

\maketitle


Let $R$ be a commutative ring and let $G$ and $H$ be finite groups. The Isomorphism Problem for group rings asks whether the existence of an isomorphism of $R$-algebras between the group rings $RG$ and $RH$ implies the existence of an isomorphism between the groups $G$ and $H$ themselves.  This problem has received special attention in the cases where $R=\mathbb Z$, the ring of integers (this  case  already appeared in G.~Higman's thesis \cite{HigmanThesis}), and where $R$ runs over all fields (this is Brauer's Problem 2 in \cite{Bra63}). Now it is known that both questions have negative answer, the second one due to an example of E.~Dade \cite{Dade71}, and the first one to one of M.~Hertweck \cite{Hertweck2001}.  Moreover, D. Passman showed that for every prime $p$ and every positive integer $n$, there are at least $p^{2(n^3-23n^2)/27}$ non-isomorphic groups with isomorphic group algebras over each field of characteristic coprime with $p$  \cite{Passman65}.

 The examples above led to the  Modular Isomorphism Problem, which  consists in the version of the Isomorphism Problem    under the additional hypotheses that $G$ and $H$ are finite $p$-groups, and $k$ a field of characteristic $p$. It was popularized by R.~Brauer  in his 1963 survey \cite{Bra63}, where it is suggested that it  ``may be   much   easier to study Problem~2 for this particular case.''   The first known partial positive result on this problem is due to W.~E.~Deskins \cite{Deskins1956}  and deals with the class of abelian finite $p$-groups. Since then, it  received considerable attention, and some of partial positive solutions were obtain by a number of authors.  For instance, the Modular Isomorphism Problem has positive solution for metacyclic groups \cite{BaginskiMetacyclic,San96}. 
 However, this problem also has negative answer in general \cite{GarciaMargolisdelRio}. An interested reader can find an almost up-to-date state of the art on the Modular Isomorphism Problem in L.~Margolis' recent survey \cite{Mar22}.

The class of $2$-generated finite $p$-groups with cyclic derived subgroup, despite its apparent simplicity, has proven to be a rich class of $p$-groups, specially regarding the Modular Isomorphism Problem: the only known indecomposable groups to fail to satisfy the statement of this problem are $2$-groups that belong to this class (see \cite{GarciaMargolisdelRio}, which contains the first known examples, and the new ones obtained in \cite{MS24, BZ24}), while for $p>2$, the situation  being quite different,   the problem is still to be decided. Our main result settles the Modular Isomorphism Problem in the positive for groups of this class under   additional constraints  on the size of the initial terms of the lower central series:
\begin{maintheorem}\label{theorem1}
	Let $p$ be an odd prime, let $k$ be the field with $p$ elements and let $G$ be a $2$-generated finite $p$-group with cyclic derived subgroup. If $kG\cong k H$ for some group $H$, then
	\begin{enumerate}
		\item  \label{theorem1.1}$ G/\gamma_3(G)^p \cong H/\gamma_3(H)^p$ and
		\item \label{theorem1.2}$ G/(G')^{p^3}\cong H/(H')^{p^3}$.
	\end{enumerate} 
\end{maintheorem}

This result fails for $p=2$ because the counter-example in \cite{GarciaMargolisdelRio} is formed by groups with derived subgroup of order $4$. 
The proof of \Cref{theorem1} is based upon a more technical result in terms of the invariants described in \cite{OsnelDiegoAngel}, that resumes the work started in \cite{GLdRS2022}. 
Namely, with the notation in \Cref{SectionApp2Gen}, we prove the following theorem. 

\begin{maintheorem}\label{theorem2}
Let $p$ be an odd prime, let $k$ be the field with $p$ elements and  let $G$ be a  $2$-generated finite $p$-group with cyclic derived subgroup and
	$$\inv(G)=(p,m,n_1,n_2,o_1,o_2,o'_1,o'_2,u_1^G,u_2^G).$$
If $k G\cong k H$ for some group $H$, then  $H$ is also a $2$-generated finite $p$-group with cyclic derived subgroup and
	$$\inv(H)=(p,m,n_1,n_2,o_1,o_2,o'_1,o'_2,u_1^H,u_2^H) $$
	such that
		$$u_2^G\equiv u_2^H \mod p, $$
	and one of the following holds:
	\begin{enumerate}
		\item \label{theorem2.1} $u_1^G\equiv u_1^H\mod p$.
		\item\label{theorem2.2} $o_1o_2>0$, $n_1+o_1'=n_2+o_2'$ and at least one of the following conditions fails:
		\begin{itemize}
			\item $u_2^G\equiv u_2^H \equiv 1 \mod p^{o_1+1-o_2}$,
			\item $n_2+o'_2=2m-o_1$.
		\end{itemize}
	\end{enumerate}
\end{maintheorem}

Observe that if $G$ and $H$ are as in the previous theorems, then  $G\cong H$ if and only if $\inv(G)=\inv(H)$, so \Cref{theorem2} is another step towards a solution of the Modular Isomorphism Problem for our target class of groups.
As an application we obtain a positive answer for the Modular Isomorphism Problem for $2$-generated $p$-groups with $p>2$ having  cyclic derived subgroup and order at most $p^{11}$. Moreover, for groups of order $p^{12}$ we also obtain a positive solution except for $p-2$ families of containing $p$ groups each. For arbitrary $p$-groups, the Modular Isomorphism Problem is known to have a positive solution for groups of order up to $p^6$ (cf. \cite[Sec. 9.3]{Mar22}).
As a byproduct of the proof of \Cref{theorem1}, be obtain the following proposition which may be of interest by itself. We do not know whether the hypothesis $p>2$ is needed.

\begin{mainproposition}\label{GG'p2}
	Let $G$ be a $2$-generated finite $p$-group with cyclic derived subgroup. Suppose that $p>2$ and $(G/G')^{p^2}$ is cyclic. If $kG\cong kH$ for some group $H$ then $G\cong H$.
\end{mainproposition}

The paper is organized as follows. In \Cref{SectionPreliminaries} we establish the notation and prove some general auxiliary results. In the remainder of the paper, $p$ is an odd prime and all the groups are $2$-generated finite $p$-groups with cyclic derived subgroup. 
In \Cref{SectionApp2Gen} we recall the classification of such groups from \cite{OsnelDiegoAngel} and establish some basic facts for these groups and their group algebras. 
In \Cref{SectionProofs} we prove Theorems~\ref{theorem1} and \ref{theorem2}.
Finally, in \Cref{SectionApplications} we prove the mentioned results about groups of small order.

\section{Preliminaries}\label{SectionPreliminaries}
 
Throughout the paper, $p$ denotes an odd prime number, $k$ is the field with $p$ elements, $G$ is a finite $p$-group and $N$ is a normal subgroup of $G$.
The group algebra of $G$ over $k$ is denoted by $kG$ and its \emph{augmentation ideal} is denoted by $\aug{G}$. It is a classical result that $\aug{G}$ is also the Jacobson ideal of $kG$.  If $C$ is a subset of $G$ then $\hat C = \sum_{c\in C} c\in kG$. It is well known that the center $\ZZ(kG)$ is the $k$-span of the class sums $\hat C$ with $C$ running on the set $\Cl(G)$ of conjugacy classes of $G$.  The rest of group theoretical notation is mostly standard: $[g,h]=g^{-1}h^{-1}gh$ for $g,h\in G$, $|G|$ denotes the order of $G$, $\ZZ(G)$ its center, $\{\gamma_i(G)\}_{i\geq 1}$ its lower central series and $G'=\gamma_2(G)$  its commutator subgroup.
For $n\geq 1$, we denote by $C_n$   the cyclic group of order $n$.  
Moreover, if $g\in G$ and $X\subseteq G$ then $|g|$ denotes the order of $g$ and $\Cen_G(X)$ the centralizer of $X$ in $G$.
For a subgroup $A$ of $G$, we denote $A^n=\GEN{a^n:x\in A}$.
If $A$ is normal cyclic subgroup of $G$, then $\aug{A^{p^n}}=\aug{A}^{p^n}$ and hence $(\augNor{A}{G})^{p^n}=\aug{A}^{p^n}kG=\augNor{A^{p^n}}{G}$.

We take the following the following notation from \cite{OsnelDiegoAngel} for integers $s,t$ and $n$ with $n \ge 0$:
	$$\Ese{s}{n} = \sum_{i=0}^{n-1} s^i.$$

We will use the following elementary lemma.

\begin{lemma}\label{ClCoset}
If $G$ is a finite $p$-group with cyclic derived subgroup and $p>2$, then every conjugacy class of
$G$ is a coset modulo a subgroup of $G'$.
\end{lemma}

\begin{proof}
Let $C$ be a conjugacy class of $G$, let $g\in C$ and $H=\{[x,g^{-1}] : x\in G\}$. Then $C=Hg$ and hence it is enough to prove that $H$ is a subgroup of $G'$. As $G'$ is cyclic and $H\subseteq G'$, it is enough to prove that if $h\in H$ then $h^i\in H$ for every non-negative integer $i$. Let $h=[x,g^{-1}]$ with $x\in G$. Then $h^x=h^r$ for some integer $r$ with $r\equiv 1 \mod p$. Therefore, using \cite[Lemma~2.1]{OsnelDiegoAngel}, we have
$[x^i,g^{-1}] = x^{-i} (x^i)^{g^{-1}} = x^{-i} (x^{g^{-1}})^i = x^{-i} (xh)^i=h^{\Ese{r}{i}}$.
This proves that $H$ contains all the elements of the form $h^{\Ese{r}{i}}$ with $i\ge 0$. By \cite[Lemma 2.2]{GLdRS2022} we deduce that $H$ contains $h^i$ for every non-negative integer.
\end{proof}

Let $n$ be a positive integer. We set
$$\Omega_n(G)=\GEN{g\in G : g^{p^n}=1}
 \qand
\Omega_n(G:N)=\GEN{g\in G: g^{p^{n}}\in N}.$$ 
Observe that $\Omega_n(G:N)$ is the only subgroup of $G$ containing $N$ such that
$$\Omega_n(G:N)/N=\Omega_n(G/N).$$

\subsection{The Jennings series}

We denote $\M_n(G)$ the $n$-th term of the \emph{Jennings series} of $G$, i.e.\
\[
\M_{n}(G) = \{g\in G : g-1\in \aug{G}^n\}= \prod_{ip^j\ge n} \gamma_i(G)^{p^j}.
\]
It is straightforward that
	\begin{equation}\label{UsefiN}
	 G\cap (1+\aug{ G}^n+\augNor{N}{G})= \M_n(G)N.
	\end{equation}
Each quotient $\M_n(G)/\M_{n+1}(G)$ is elementary abelian and, if $t$ is the smallest non-negative integer with $\M_{t+1}(G)=1$, then a \emph{Jennings set} of $G$ is a subset $\{g_{11},\dots,g_{1d_1},g_{21},\dots,g_{2d_2},\dots|g_{t1},\dots,g_{td_t}\}$ of $G$ such that $g_{i1}\M_{i+1}(G), \dots, g_{id_i}\M_{i+1}(G)$ is a basis of $\M_n(G)/\M_{n+1}(G)$ for each $i$. Observe that $|G|=p^{\sum_{i=1}^t d_i}$.
If $x_1,\dots,x_n$ are the elements of a Jennings set of $G$, in some order,
then
	$$\mathscr B = \{(x_1-1)^{e_1}\cdots (x_n-1)^{e_n} : 0\le e_i \le p-1 \text{ and }\sum_{i=1}^n e_i >0\}$$
is a basis of $\aug{G}$, called a \emph{Jennings basis} of $\aug{G}$ associated to the given Jennings set. We denote $\mathscr B^n = \mathscr B \cap \aug{G}^n$, which is a basis of $\aug{G}^n$.

\begin{lemma}\label{lemma:JTN}
There is a Jennings set $\mathscr S$ of $G$ such that $N\cap \mathscr S$ is a Jennings set of $N$.
\end{lemma}

\begin{proof}
We argue by induction on $|N|$.  If $|N|=1$, then there is nothing to prove.
Now suppose that the result holds for normal subgroups of order $p^n$, and assume that $N$ has order $p^{n+1}$.
Since $G$ is a $p$-group, the center of $G$ intersects $N$ non-trivially, so we can choose a subgroup $L\subseteq N\cap \ZZ(G)$ of order~$p$.
By the induction hypothesis, we can choose a Jennings set $\bar {\mathscr S}$ of $G/L$ such that $\bar  {\mathscr S}\cap (N/L)$ is a Jennings set of $N/L$. Let $\mathscr S$ be a set of representatives of the elements of $\bar {\mathscr S}$ in $G$. Clearly, the representatives of elements in $N/L$ are in~$N$. For some $i$ we have that  $L \subseteq D_i(G)$  but $L \not \subseteq D_{i+1}(G)$, and for some $j$, that $   L\subseteq \M_j (N)$ but $L\not\subseteq \M_{j+1}(N)$. Observe that $\mathscr S$ is almost a Jennings basis of $G$ except it does not contain representatives of a basis of $\M_i(G)/\M_{i+1}(G)$, only of a maximal linear subspace which is a direct complement of $L$. Similarly, $\mathscr S\cap N$ is almost a Jennings basis of $N$ except it does not contain representatives of a basis of $\M_j(N)/\M_{j+1}(N)$, only of a maximal linear subspace which is a direct complement of $L$. Hence it suffices to take the Jennings set  $\mathscr S\cup \{l\}$, where $l$ is a generator of $L$.
\end{proof}

The following equality is \cite[Theorem~A]{Usefi2008} and its symmetric analogue:
\begin{equation}\label{Usefi}
	\M_{n+1}(N)=G\cap (1+\aug{N}^n\aug{G})=G\cap (1+\aug{G}\aug{N}^n).
\end{equation}
It can be generalized as follows.

\begin{lemma}\label{UsefiFrattini}
If $n$ and $m$ are  positive integers, then $$(1+\aug{G}^n+\aug{N}^{m }\aug{G})\cap G = \M_n(G)\M_{m+1}(N)=(1+\aug{G}^n+\aug{G}\aug{N}^{m })\cap G .$$
\end{lemma}

\begin{proof}
We  prove only the first identity, the second being analogous.
Since $(1+\aug{G}^n)\cap G=\M_{n}(G)$ and
$(1+\aug{N}^m\aug{G})\cap G \supseteq (1+\aug{N}^{m+1}) \cap G= \M_{m+1}(N)$, the right-to-left inclusion is clear.
Thus it suffices to prove the converse.
Taking quotients modulo $\M_n(G)\M_{m+1}(N)$, it is enough to prove that
	\begin{equation}\label{UsefiFrattini:AngelLaVaAQuitar}
\M_n(G)\M_{m+1}(N)=1\qquad \text{implies}\qquad (1+\aug{G}^n+\aug{N}^{m }\aug{G})\cap G=1.
	\end{equation}

By \Cref{lemma:JTN}, there is a Jennings set $\mathscr S$ of $G$ such that $N\cap \mathscr S$ is a Jennings set of $N$.
Ordering the elements of $\mathscr S$ so that those in $N$ are placed first we obtain a Jennings basis $\mathscr B$ of $\aug{G}$ associated to $\mathscr S$ containing a Jennings basis $\mathscr B_0$ of $\aug{N}$ associated to $N\cap  \mathscr S$.
Recall that the set $\mathscr B^n= \mathscr B\cap \aug{G}^n$ is a basis of $\aug{G}^n$. Moreover, the set $\mathscr B_0^m=\mathscr B\cap \aug{N}^m\aug{G}$ is a basis of
$\aug{N}^m\aug{G}$, and coincides with the set of elements of $\mathscr B$
of the form $xy$ with $x\in \mathscr B_0\cap \aug{N}^m$ and $y\in \aug{G}$.
Then the following implication is clear: if $y\in \mathscr B$ occurs in the support in the basis $\mathscr B$ of an element $x\in \aug{G}^n+\aug{N}^m\aug{G}$, then $y\in \mathscr B^n\cup \mathscr B_0^m$.

	Moreover, it is clear $(1+\mathscr B^n)\cap G\subseteq (1+\aug{G}^n)\cap G=\M_n(G)$ and $(1+\mathscr B_0^m)\cap G\subseteq (1+\aug{N}^m\aug{G})\cap G= \M_{m+1}(N)$ by \eqref{Usefi}.  Thus $(1+\mathscr B^n\cup \mathscr B_0^m)\cap G \subseteq \M_n(G) \M_{m+1}(N)$.

We prove \eqref{UsefiFrattini:AngelLaVaAQuitar} by induction on $m$.  Suppose first that  $m=1$ and that $\M_n(G)\M_2(N)=1$, so   \eqref{UsefiN} yields
	$$(1+\aug{G}^n+\aug{N}^{   }\aug{G})\cap G \subseteq
	(1+\aug{G}^n+\augNor{N}{G}  )=\M_n(G)N=N.$$
So, if $1\neq g\in (1+\aug{G}^n+\aug{N}^{     }\aug{G})\cap G$, then $g\in N$. Since $N$ is elementary abelian, $g-1\in \aug{N}\setminus \aug{N}^2$. Thus the support of $g-1$ in the basis $\mathscr B_0$ contains an element of the form $h-1$, with $1\neq h\in N $. Then, by the two previous paragraphs, $h\in (1+ \mathscr B^n\cup \mathscr B _0^1)\cap G\subseteq \M_n(G)\M_{2}(N)=1$, a contradiction.

	For $m> 1$, the induction step is similar. Suppose that $\M_n(G)\M_{m+1}(N)=1$, so $\M_m(N)$ is elementary abelian. Take
	$$1\neq g\in  (1+\aug{G}^n+\aug{N}^{ m  }\aug{G})\cap G \subseteq (1+\aug{G}^n+\aug{N}^{m-1}\aug{G}  )=\M_n(G)\M_{m }(N)=\M_m(N). $$
	Since $\mathscr B_0\cap \aug{\M_{m  }(N)}$ is a Jennings basis of $\aug{\M_m(N)}$ and $g-1\in \aug{\M_m(N)}\setminus \aug{\M_m(N)}^2$, we have that the support of $g-1$ in this basis (and hence in the basis $\mathscr B $) contains an element of the form $h-1$, with $1\neq h\in \M_m(N)$.  However, $h\in (1+\mathscr B^n\cup \mathscr B_0^m)\subseteq \M_n(G)\M_{m+1}(N)=1$, a contradiction.
\end{proof}

\subsection{The relative lower central series}

The \emph{lower central series of $N$ relative to $G$} is the series defined recursively by
$$\gamma_1^G(N)=G \qand
 	\gamma_{n+1}^G(N)=[ \gamma_n^G(N),N ].$$
 We consider also the sequence of  ideals of $kG$ defined recursively by setting 
$$ \Aug{1}{N,G} =   \aug{N}\aug{G} \qand
 \Aug{+1}{N,G}=\aug{N}\Aug{i}{N,G} +\Aug{i}{N,G}  \aug{N}.$$
This can be also defined with a closed formulae: 
 \begin{equation}\label{JClosed}
 \Aug{n}{N,G}=\aug{N}^n\aug{G}+\sum_{i=1}^{n-1} \aug{N}^{n-i}\aug{G}\aug{N}^{i}.
 \end{equation}

From $\aug{N}kG=kG\aug{N}$ and \eqref{JClosed} it easily follows that 
\begin{equation}\label{IJISandwich}
\aug{N}^n\aug{G}\subseteq \Aug{n}{N,G}\subseteq \aug{N}^n kG.
\end{equation}

\begin{lemma}\label{WellDefined}
The following is a well defined   map:
  $$\Lambda_N^n=\Lambda^n_{N,G}:\frac{\augNor{N}{G} }{\aug{N}\aug{G}} \longrightarrow
  \frac{\aug{N}^{p^{n}}kG}{\Aug{p^{n}}{N,G}}, \qquad
x+\aug{N}\aug{G}\mapsto x^{p^n}+\Aug{p^{n}}{N,G}.$$
\end{lemma} 	

\begin{proof}
Let $x\in \augNor{N}{G} $ and $y\in \aug{N}\aug{G}$. Then $(x+y)^{p^n}-x^{p^n}  = \sum_i a_i$ where each $a_i$ is a product of $p$ elements of $\{x,y\}$ with at least one equal to $y$. Hence each $a_i\in I_1\dots I_{p^n}$, where each $I_i$ is either $\augNor{N}{G} $ or $\aug{N} \aug{G}$,  and at least one of the $I_i$'s is of the second type. Since $\aug{N} \aug{G}\subseteq \augNor{N}{G}$, $I_1\dots I_{p^n}\subseteq I(N)^{p^n-j} I(G) I(N)^j$ for some $0\leq j\leq p^n$, and hence, by  \eqref{JClosed}, $I_1\dots I_{p^n}\subseteq \Aug{p^{n}}{N,G}$. Therefore $(x+y)^{p^n}-x^{p^n} \in \Aug{p^{n}}{N,G}$, so $\Lambda^n_N$ is well defined.
 \end{proof}

The ambient group $G$ will be always clear from the context so we just write $\Lambda^n_N$.
 	In particular, 
 	$$\Lambda_G^n: \frac{\aug{G}}{\aug{G}^2}\to \frac{\aug{G}^{p^n}}{\aug{G}^{p^{n}+1}}$$
 	is the usual map used in the \emph{kernel size} computations (see \cite{Passman1965p4}).

 The first statement of the next lemma is just a slight modification of a well-known identity (see \cite[Lemma 2.2]{San89}), while the second one is inspired, together with the definition of the ideals $\Aug{i}{N,G}$, by the first section of \cite{BC88}.   For the convenience of the reader  we include a proof. 

 \begin{lemma}\label{LemmaJ}
 	Let $L$ and $N$ be normal subgroups of $G$. Then the following equations hold
\begin{eqnarray}
\label{eq71}
  \aug{L}  \augNor{N}{G}+\aug{N} \augNor{L}{G}  &=& \augNor{[L,N]}{G} +\aug{N}\augNor{L}{G}, \\
\label{eq72}
 	\Aug{n}{N,G}&=& \sum_{i=1}^{n }  \aug{N}^{n+1-i}    \augNor{\gamma_i^G(N)}{G}  .
\end{eqnarray}
 \end{lemma}
 
 \begin{proof}
 Since the  terms at both sides of \eqref{eq71} are  two-sided  ideals  of $kG$, the equation follows from
 	$$(g-1)(h-1)=hg([g,h]-1)+(h-1)(g-1)\qquad \text{for }g,h\in G.$$

In order to prove \eqref{eq72} we proceed by induction on $n$. For $n=1$ there is nothing to prove, and the following chain of equations
 	\begin{eqnarray*}
 		\Aug{n+1}{N,G} &=&\Aug{n}{N,G}\aug{N} + \aug{N} \Aug{n}{N,G}\\
 		& =&  \sum_{i=1}^n \aug{N}^{n+1-i}  \augNor{\gamma_i^G(N)}{G}  \aug{N} +
 		\aug{N}\sum_{i=1}^n \aug{N}^{n+1-i} \augNor{\gamma_i^G(N)}{G}  \\
 		&= &  \sum_{i=1}^n \aug{N}^{n+1-i} \left[ \aug{\gamma_i^G(N)} \augNor{N}{G} +
 		\aug{N} \augNor{\gamma_i^G(N)}{G}\right] \\
 		\text{(by   \eqref{eq71}  with  $L=\gamma_i^G(N)$)}	&=& 
 		\sum_{i=1}^n \aug{N}^{n+1-i} \left(\augNor{\gamma_{i+1}^G(N)}{G}    +   \aug{N}  \augNor{\gamma_i^G(N)}{G}   \right) \\
 		&=& \sum_{i=1}^{n+1} \aug{N}^{n+2-i} \augNor{\gamma_i^G(N)}{G}
 	\end{eqnarray*} 
 completes the induction argument.
 \end{proof}

\begin{lemma}\label{JotasLema}
Let $N$ be a normal subgroup of $G$.
\begin{enumerate}
\item If $\gamma_i^G(N)\subseteq \M_i(N)$ for every $i\ge 2$ then for every $n\ge 1$ we have $\Aug{n}{N,G}= \aug{N}^n\aug{G}$. 
\item If $[G,N]\subseteq N^p$ then $\gamma_i^G(N)\subseteq \M_i(N)$ for every $i\ge 2$.
\end{enumerate}
\end{lemma}

\begin{proof}
(1) Suppose that $\gamma_i^G(N)\subseteq \M_i(N)$ for $i\ge 2$. 
Since $\M_i(N) \subseteq   1+\aug{N}^i$, it follows that if $i\ge 2$ then 
$\aug{\gamma_i^G(N)}\subseteq \aug{N}^i$ and hence, using \eqref{eq72}  we have 
$$\Aug{s}{N,G}=\aug{N}^s\aug{G} + \sum_{i=2}^s \aug{N}^{s+i-1}\augNor{\gamma_i^G(N)}{G}\subseteq \aug{N}^s\aug{G} + \aug{N}^{s+1}kG \subseteq \aug{N}^s\aug{G}.$$
This, together with \eqref{IJISandwich},  completes the proof. 

(2) Suppose that $[G,N]\subseteq N^p$. Then $\gamma_2^G(N)=[G,N]\subseteq N^p\subseteq \M_2(N)$. Then arguing by induction on $i$, for every $i\ge 3$ we  obtain   $\gamma_i^G(N)=[\gamma_{i-1}^G(N),N]\subseteq
[\M_{i-1}(N),\M_1(N)] \subseteq \M_i(N)$, because $(\M_i(N))_i$ is an $N_p$-series. 
\end{proof}

\subsection{Canonical subquotients and maps}

Let $\G$ be a class of groups.
Roughly speaking, we say that a certain assignation defined on $\G$ is canonical if it ``depends only on the isomorphism type of $kG$ as $k$-algebra''.
More precisely, suppose that for each $G$ in $\G$ we have associated a subquotient $U_G$ of $k G$ as $k$-space. 
We say that $G\mapsto U_G$ is canonical in $\G$ if every isomorphism $k$-algebras $\psi:kG\to kH$, with $G$ and $H$ in $\G$, induces an isomorphism
$\tilde{\psi}:U_G\mapsto U_H$ in the natural way.
If $(G\mapsto U_G^{(x)})_{x\in X}$ is a family of canonical subquotients in $\G$ then we also say that $G\mapsto \prod_{x\in X} U_G^{(x)}$ is canonical in $\G$. In this case every isomorphism $\psi:kG \to kH$ with $G$ and $H$ in $\G$ induces an isomorphism $\prod_{x\in X} U_G^{(x)}\to \prod_{x\in X} U_H^{(x)}$ in the natural way.

\begin{lemma}\label{lemma:canonicalsubgroups}
The following assignations are canonical in the class of $p$-groups:
\begin{itemize}
\item $G\mapsto \augNor{\Omega_n(G:G')}{G}$.
\item $G\mapsto \augNor{\Omega_n(G:\ZZ(G)G')}{G}$.
\end{itemize}
\end{lemma}

\begin{proof}
See \cite[Proposition 2.3(1) and Lemma 3.6]{GLdRS2022}.
\end{proof}

\begin{lemma}\cite[Theorem 4.2(1)]{GLdRS2022}\label{lemma:canonicalsubgroups2}
The assignation $G\mapsto \augNor{\Cen_G(G')}{G}$ is canonical in the class of $p$-groups with cyclic derived subgroup and $p$ odd.
\end{lemma}

We note that, if  $G\mapsto \augNor{N_G}{G}$ is canonical in $\G$, where $N_G$ is a normal subgroup of $G$, then an easy induction  on $n$ shows that $G\mapsto \Aug{n}{N_G,G})$ is canonical in $\G$ too.

Now suppose that for each $G$ in $\G$ we have associated a map $f_G: U_G \to V_G$, with $U$ and $V$ products of canonical subquotients in $\G$.
We say that $G\mapsto f_G$ is \emph{canonical} if for every isomorphism $\psi :kG\to k H$ the following square is commutative
$$\xymatrix{
	U_G \ar[d]_-{\tilde \psi} \ar[r]^-{f_G} & V_G \ar[d]^-{\tilde  \psi} \\
	U_H \ar[r]_-{f_H} & V_H 
} $$
For example, the assignation $G\mapsto \Lambda_G^n$ described above is canonical in the class of finite $p$-groups, and so is $G\mapsto \Delta_G$, where $\Delta_G$ is the natural projection:
$$\Delta_G:\frac{\augNor{G'}{G}}{\aug{G'}\aug{G}} \longrightarrow \frac{\augNor{G'}{G}+\aug{G}^3}{\aug{G}^3}, \quad x+ \aug{G'}\aug{G} \mapsto x+ \aug{G}^3.$$
Observe that $\Delta_G$ is well defined homomorphism of $k$-algebras because $\aug{G'}\subseteq \aug{G}^2$.

In order to simplify notation, instead of writing ``$G\mapsto A_G$ is canonical'' we just write ``$A_G$ is canonical'', where $A_G$ is either a product of subquotients or a map between canonical products of subquotients. 

For mnemonic purposes we use variations of the symbols $\Lambda^n$ and $\Upsilon^n$ for maps of the kind $x\mapsto x^{p^n}$.
Moreover we will encounter a number of projection maps of the kind $x+I\mapsto x+J$ for ideals $I\subseteq J$, for which we use variations of the symbols $\Delta, \zeta$ and $\nu$, with the hope they help the reader to recall the domain: $\Delta$ refers to derived subgroup, $\zeta$ to center and $\nu$ to some normal subgroup $N$. Other projection maps are denoted with variations of $\pi$ and $\eta$.

\section{2-generated finite $p$-groups with cyclic derived subgroup}\label{SectionApp2Gen}

The non-abelian 2-generated finite $p$-groups with cyclic derived subgroup have been classified in \cite{OsnelDiegoAngel} in terms of numerical invariants.
For the reader's convenience, we include in the following theorem a simplification of this classification for the case $p>2$.

\begin{theorem}[\cite{OsnelDiegoAngel}]\label{Main}For a list of non-negative integers $I=(p,m,n_1,n_2,o_1,o_2,o'_1,o'_2,u_1,u_2)$ where $p>2$ is a prime number,  let  $\G_I$ be the group  defined by 
	$$\G_I=\GEN{b_1,b_2,a=[b_2,b_1] \mid a^{p^m}=1, a^{b_i}=a^{r_i}, b_i^{p^{n_i}}=a^{u_ip^{m-o'_i}}},$$
	where 	\begin{equation}\label{Erres}
		r_1=1+p^{m-o_1} \qand 
		r_2 = \begin{cases} 1+p^{m-o_2}, & \text{if } o_2>o_1; \\ r_1^{p^{o_1-o_2}}, & \text{otherwise}. \end{cases}
	\end{equation}
Then $I\mapsto [\G_I]$, where $[\G_I]$ denotes the isomorphism class of $\G_I$, defines a bijection  between  the set of lists of integers $(p,m,n_1,n_2,  o_1,o_2,o'_1,o'_2,u_1,u_2)$ satisfying conditions \ref{1}-\ref{6}, and the isomorphism classes of $2$-generated non-abelian groups of odd prime-power  order with cyclic derived subgroup.
	\begin{enumerate}[label=$(\Roman*)$] 
		\item \label{1} $p$ is prime and $n_1\geq n_2 \ge 1 $.
		\item \label{2}  $0\le o_i<\min(m,n_i)$, $0\le o'_i \le m-o_i$ and $p\nmid u_i$ for $i=1,2$.
		\item \label{3} One of the following conditions holds:  
		\begin{enumerate}[label=$(\alph*)$] 
			\item $o_1=0$ and $o'_1\le o'_2\le o'_1+o_2+n_1-n_2$.  
			\item $o_2=0<o_1$, $n_2<n_1$ and $o'_1+\min(0,n_1-n_ 2  -o_1)\le o'_2\le o'_1+n_1-n_2$.
			\item $0<o_2< o_1<o_2+n_1-n_2$ and $o'_1\le o'_2\le o'_1+n_1-n_2$. 
		\end{enumerate}	
		
		\item \label{4}   $o_2+o_1'\leq m\le n_1$  and one of the following conditions hold:
		\begin{enumerate}[label=$(\alph*)$] 
			\item $o_1+o'_2\le m \le n_2$.
			\item $2m-o_1-o'_2=n_2<m$ and $u_2\equiv 1 \mod p^{m-n_2}$.		
		\end{enumerate}    
		
		\item \label{5}$ 1\le u_1  \leq p^{a_1}$, where 
		$a_1=\min(o'_1,o_2+\min(n_1-n_2+o'_1-o'_2,0)).$
		\item \label{6}One of the following conditions holds:
		\begin{enumerate}[label=$(\alph*)$] 
			\item $ 1\le  u_2 \leq p^{a_2}$.
			\item $o_1o_2\neq 0$, $n_1-n_2+o_1'-o_2'=0<a_1$,  $1+p^{a_2}\leq u_2 \leq 2p^{a_2}$, and $u_1 \equiv 1\mod p$;  
		\end{enumerate}
		where  
		$$a_2=  \begin{cases}
			0, &\text{if } o_1=0; \\
			\min(o_1,o'_2,o'_2-o'_1+\max(0,o_1+n_2-n_1)), & \text{if } o_2=0<o_1; \\ \min(o_1-o_2,o'_2-o'_1), & \text{otherwise.} \end{cases}$$
	\end{enumerate}
\end{theorem}

For every non-abelian 2-generated finite $p$-group $\Ge$ with cyclic derived subgroup and $p$ odd, let $\inv(\Ge)$ denote the unique list satisfying the conditions of the previous theorem such that $\Ge$ is isomorphic to $\G_{\inv(\Ge)}$.
An explicit description of $\inv(\Ge)$ can be found in \cite{OsnelDiegoAngel} and also in \cite{GLdRS2022}. In these references the list $\inv(\Ge)$ has two additional entries
$\sigma_1$ and $\sigma_2$ which for $p>2$ always equal 1, so we drop them.

In this section $\Ge$ is a $2$-generated finite $p$-group with cyclic derived subgroup, and we set
 $$\inv(\Ge)=(p,m,n_1,n_2,o_1,o_2,o'_1,o'_2,u_1,u_2).$$
Hence $\Ge$ is given by the following presentation 
	$$\Ge=\GEN{b_1,b_2 \mid a=[b_2,b_1], a^{b_i}=a^{r_i}, b_i^{p^{n_i}}=a^{u_i p^{m-o'_i}}},$$
where $r_1$ and $r_2$ are as in \eqref{Erres}.
By \cite[Lemma~3.5]{GLdRS2022}, 
\begin{equation}\label{gamma}
\gamma_n(\GG) = \GEN{a^{p^{(n-2)(m-\max(o_1,o_2))}}}, \text{ for } n\ge 2.
\end{equation}
In particular $[\Ge,\Ge']=\gamma_3(\Ge)\subseteq \GEN{a^p}=(\Ge')^p$, and hence, by \Cref{JotasLema},
	$$\Aug{n}{\Ge',\Ge}=\aug{\Ge'}^n\aug{\Ge} \text{ for every } n\ge 1.$$ 
By \cite[Lemma~2.2]{GLdRS2022}, there is a unique integer $\delta$ satisfying
\begin{equation}\label{eq:CenteredCongruence}
1\le \delta \le p^{o_1} \qand \Ese{r_2}{\delta p^{m-o_1}}\equiv -p^{m-o_1}\mod p^m.
\end{equation} 		
Moreover, $p\nmid \delta$.
By \cite[Lemma 3.7]{GLdRS2022}
\begin{equation}\label{eq:centerGens}
\ZZ(\Ge)=\GEN{b_1^{p^m}, b_2^{p^m}, c}, \quad \text{where }c=\begin{cases}
b_1^{\delta p^{m-o_2}}a, & \text{if }o_1=0; \\
b_1^{-\delta p^{m-o_2} } b_2^{\delta p^{ m-o_1}} a, & \text{otherwise}.
\end{cases}
\end{equation}

Observe that 
\begin{equation}\label{nMenorni}
n< n_i \quad \text{implies} \quad b_i^{p^n}\not\in \M_{p^n+1}(\Ge)\Ge', \text{ for } i=1,2.
\end{equation}
Furthermore, for every $n\ge 0$,
\begin{equation}\label{Mpn}
\M_{p^n}(\GG)=\GG^{p^n}.
\end{equation}
To prove this t suffices to show that $ip^j\ge p^n$ implies $\gamma_i(\GG)^{p^j}\subseteq \Gamma^{p^n}$.
This is clear if $j\ge n$.
Otherwise, $j<n$, $i\ge 2$ and $i-2\ge p^{n-j}-2\ge n-j$, since $p\ge 3$.
Using \eqref{gamma} we obtain that $\gamma_i(\GG)^{p^j}=\GEN{a^{p^{j+(i-2)(m-\max(o_1,o_2))}}} \subseteq \GEN{a^{p^n}} \subseteq \GG^{p^n}$. Thus \eqref{Mpn} follows.

Moreover, 
 \begin{equation} \label{eq:n1=m}
 	n_1=m \qquad \text{implies} \qquad  o_1o_2=0.
\end{equation}
 To see this, observe that if $o_1o_2>0$ and $n_1=m$ then $m>n_2$ by condition \ref{3}, so $n_2=2m-o_1-o_2'$ by condition \ref{4}. Thus, by conditions \ref{2} and \ref{3},  $o_1-o_2< n_1-n_2=o_1+o_2'- m \leq o_1-o_2$, a contradiction.

In the rest of this section we assume the following:
	\begin{equation}\label{Assumptions}
	o_1\ne o_2, \quad 0<\max(o'_1,o'_2)<m \qand n_2\ge 2.
	\end{equation}
In the next section we will see that this is the only case of interest, as if any of these conditions fails, then the Modular Isomorphism Problem has a positive solution for $\Gamma$.
	
Observe that if $n<m-1$ then $\aug{\Ge'}^{p^n}k\Ge/\aug{\Ge'}^{p^n}\aug{\Ge}$ is a one-dimensional $k$-space generated by the class of $a^{p^n}-1$.
Moreover the image of
$\Delta_\Ge$ is spanned by $a-1+ \aug{\Ge}^3$.
As $p$ is odd, $\Ge^p = \M_3(\Ge)$, and as $\max(o'_1,o'_2)<m$,
$a\not\in \Ge^p$. Thus $a-1\not\in \aug{\Ge}^3$. Then, we have the following

\begin{lemma}
$\Delta_\Ge$ is an isomorphism.
 \end{lemma}

\begin{lemma}\label{WeightClassSums}
 $\hat C\in\aug{\Ge}^{(p-1)p^m}$ for each non-central conjugacy class $C$ of $\Ge$.
\end{lemma}
\begin{proof}
By hypothesis $o_i'>0$ for some $i\in \{1,2\}$. 
In that case $m \leq   n_i+o_i'-1$, by condition \ref{4}. Thus, it is enough to show that if $o'_i>0$, then $\hat C\in \aug{\Ge}^{(p-1)p^{n_i+o_i'-1}}$. 
	
If $x$ is an indeterminate over $k$ and $n\ge 1$ then we have 
	$$\sum_{i=1}^{p^n-1} x^i = \frac{x^{p^n}-1}{x-1} = (x-1)^{p^n-1}.$$ 
	Hence, using \Cref{ClCoset}  for each $C\in \Cl(\Ge)$ such that $|C|>1$, and $g\in C$, there exists $0\leq n<m$ such that
	\begin{eqnarray*}
		\hat C= \sum_{i=0}^{p^{m-n}-1} a^{ip^n}g= (a^{p^n}-1)^{p^{m-n}-1} g&=&(a^{p^n}-1)^{(p-1)p^{m-n-1} }  (a^{p^n}-1)^{ p^ {m-n-1}-1}g  \\ 
		& =&(a^{p^{m-1}}-1)^{(p-1)} (a^{p^n}-1)^{p^{m-n-1}-1}g,
	\end{eqnarray*}
	and this element belongs to $\aug{\Ge}^{ (p-1)p^{n_i+o_i'-1} }$, as   the  hypothesis  $o'_i>0$ implies
	$$a^{p^{m-1}}=b_i^{p^{n_i+o_i'-1}}\in \M_{p^{n_i+o_i'-1}}(\Ge).$$
\end{proof}

In the remainder of the section we consider a series of subquotients of $k\Ge$ and maps which, by construction, are canonical in the class of 2-generated finite $p$-groups with cyclic derived subgroup satisfying \eqref{Assumptions}, and will play a central rôle in the proof of our main results.

Recall from \cite[Lemma~6.10]{Sandling85} that
\begin{equation}\label{eq:center}
	\ZZ(\aug{\Ge})=\aug{\ZZ(\Ge)}\oplus \left( \bigoplus_{C\in \Cl(\Ge),|C|>1} k \hat C \right).
\end{equation}  

Observe that as $o_i<m$ for $i=1,2$, $c\in \M_2(\Ge)$, where $c$ is as in \eqref{eq:centerGens}, hence $c-1\in \aug{\Ge}^2$.
Then \Cref{WeightClassSums} and \eqref{eq:center} yield
\begin{eqnarray}\label{eq:killcc}
\begin{split}
	\frac{\ZZ(I( \Ge))+ \aug{\Ge}^{p^m}}{\aug{\Ge}^{p^m}} &=
	\frac{\aug{\ZZ(\Ge)}+ \aug{\Ge}^{p^m}}{\aug{\Ge}^{p^m}} \\
	&=
	\frac{ k(c-1)+ k(c-1)^2+\dots + k(c-1)^{\frac{p^{m}-1}{2}} +\aug{\Ge}^{p^{m}}}{\aug{\Ge}^{p^{m}}}.
\end{split}
\end{eqnarray}
Hence, 
$$ \frac{ \ZZ( \aug{\Ge})+\aug{\Ge}^{3 }}{\aug{\Ge}^{3}}=\frac{k(a-1)+\aug{\Ge}^{3 }}{\aug{\Ge}^{3 }} $$
since $c-a\in \aug{\Ge}^3$, and, for $o=\max(o_1,o_2)$,
\begin{equation}\label{ZIGpm}
\frac{ \ZZ( \aug{\Ge})+\aug{\Ge}^{p^{m-o}+1} +\augNor{\Ge'}{\Ge}}{\aug{\Ge}^{p^{m-o}+1}+\augNor{\Ge'}{\Ge}}= \begin{cases}
\frac{ k(b_1^{  p^{m-o_2}}-1)+\aug{\Ge}^{p^{m-o_2}+1} +\augNor{\Ge'}{\Ge}}{\aug{\Ge}^{p^{m-o_2}+1}+\augNor{\Ge'}{\Ge}}, & \text{if }o_1=0; \\
\frac{ k(b_2^{  p^{m-o_1}}-1)+\aug{\Ge}^{p^{m-o_1}+1} +\augNor{\Ge'}{\Ge}}{\aug{\Ge}^{p^{m-o_1}+1}+\augNor{\Ge'}{\Ge}}, &\text{if }o_1\ne 0.
\end{cases}
\end{equation}
This subquotient of $k \Ge$ is one-dimensional by \eqref{nMenorni} and \eqref{UsefiN}. 

Then we consider the canonical maps
$$\zeta_\Ge^1: \frac{ \ZZ( \aug{\Ge})+\aug{\Ge}^{p^{m}}}{\aug{\Ge}^{p^{m}}} \to \frac{ \ZZ( \aug{\Ge})+\aug{\Ge}^{3 }}{\aug{\Ge}^{3 }}, \ w+\aug{\Ge}^{p^m}\mapsto w+ \aug{\Ge}^3,$$
and 

$$ \zeta_\Ge^2: \frac{ \ZZ( \aug{\Ge})+\aug{\Ge}^{p^{m}}}{\aug{\Ge}^{p^{m}}} \to \frac{ \ZZ( \aug{\Ge})+\aug{\Ge}^{p^{m-o}+1} +\augNor{\Ge'}{\Ge}}{\aug{\Ge}^{p^{m-o}+1}+\augNor{\Ge'}{\Ge}}, \ w+\aug{\Ge}^{p^{m}}\mapsto w+\aug{\Ge}^{p^{m-o}+1}+\augNor{\Ge'}{\Ge}.$$
It is immediate that for $x_1,\dots,x_{(p^m-1)/2}\in k$, 
$$\zeta_\Ge^1\left(\sum_{i=1}^{\frac{p^m-1}{2}} x_i(c-1)^i+\aug{\Ge}^{p^m}\right)= x_1(a-1)+\aug{\Ge}^3 $$
and 
$$\zeta_\Ge^2\left(\sum_{i=1}^{\frac{p^m-1}{2}} x_i(c-1)^i+\aug{\Ge}^{p^m}\right)=\begin{cases}
x_1 (b_1^{p^{m-o_2}}-1)+\aug{\Ge}^{p^{m-o}+1}+\augNor{\Ge'}{\Ge}, &\text{if }o_1=0;\\
x_1\delta(b_2^{p^{m-o_1}}-1)+\aug{\Ge}^{p^{m-o}+1}+\augNor{\Ge'}{\Ge}, & \text{if }o_1\ne 0.
\end{cases}$$
The first implies that $\Imagen(\zeta^1_\Ge)=\Imagen(\Delta_\Ge)$.

For each $n\geq 1$ let  
	$$\mathcal{C}_\Ge=\frac{\augNor{C_\Ge(\Ge')}{\Ge}+\aug{\Ge}^2}{\aug{\Ge}^2}= \begin{cases}
 	\frac{k(b_1-1)+\aug{\Ge}^2}{\aug{\Ge}^2}, &\text{if }o_1=0; \\
 	\frac{k(b_2-1)+\aug{\Ge}^2}{\aug{\Ge}^2}, & \text{if } o_1\ne 0.
 	\end{cases} $$ 
Then
\begin{equation}\label{ImagenLambda}
\Lambda_{\Ge}^{n}(\mathcal{C}_\Ge)= \begin{cases}\frac{k(b_1-1)^{p^{n}}+ \aug{\Ge}^{p^{n}+1}}{\aug{\Ge}^{p^{n}+1}}, &\text{if }o_1=0; \\
\frac{k(b_2-1)^{p^n}+ \aug{\Ge}^{p^n+1}}{\aug{\Ge}^{p^n+1}},&\text{if }o_1\ne 0.
\end{cases}
\end{equation}
Let $\tilde \Lambda_{\Ge }^{n}:\mathcal{C}_\Ge\to \Lambda^n_\Ge(\mathcal{C}_\Ge)$ be the  restriction of $\Lambda_{\Ge }^{n}$ to $\mathcal{C}_\Ge$. 
By \eqref{nMenorni}, 
\begin{equation}\label{LambdaTildeIso}
\text{if either $o_1=0$ and $n<n_1$ or $o_1\ne 0$ and $n<n_2$, then $\tilde\Lambda^n_\Ge$ is an isomorphism.}
\end{equation} 
Observe that $m-o<n_i$ for $i=1,2$. Indeed, if $m-o\ge n_i$ then, as $o>0$ and $o'_2<m$, by condition \eqref{Assumptions}, $i=2$ and $n_2=2m-o_1-o'_2> m-o_1\ge m-o$, a contradiction. 
Thus $\tilde\Lambda^{m-o}_\Ge$ is an isomorphism and hence $\Lambda^{m-o}_\Ge(\mathcal{C}_\Ge)$ is one-dimensional.
Therefore we have isomorphisms
	\begin{eqnarray}\label{CG}
	\mathcal C_\Ge 
	\stackrel{\tilde \Lambda^{m-o}_\Ge}{\longrightarrow} 
	\Lambda^{m-o}_\Ge(\mathcal{C}_\Ge) 
	\stackrel{\pi_\Ge}{\longrightarrow} \frac{ \ZZ( \aug{\Ge})+\aug{\Ge}^{p^{m-o}+1} +\augNor{\Ge'}{\Ge}}{\aug{\Ge}^{p^{m-o}+1}+\augNor{\Ge'}{\Ge}}
	\end{eqnarray}
where $\pi_\Ge$ is another natural projection, i.e.  
	$\pi_\Ge\left(x+\aug{\Ge}^{p^{m-o}+1}\right) = x+\aug{\Ge}^{p^{m-o}+1}+\augNor{\Ge'}{\Ge}$.

\section{Proof of the main results}\label{SectionProofs}

Recall that $p$ is an odd prime integer and $k$ the field with $p$ elements. 
For the remainder of the paper, we fix the following notation. Let $G$ denote a $2$-generated finite $p$-group with cyclic derived subgroup, let $H$ denote another group and let  $\psi:kG\rightarrow kH$ be an isomorphism of $k$-algebras. 
By \cite[Theorem~C]{GLdRS2022}, $H$ is $2$-generated with cyclic derived subgroup, and $\inv(G)$ and $\inv(H)$ coincide in all but the last entries. 
So we may write
	$$\inv(G)=(p,m,n_1,n_2,o_1,o_2,o'_1,o'_2,u_1^G,u_2^G) \qand
\inv(H)=(p,m,n_1,n_2,o_1,o_2,o'_1,o'_2,u_1^H,u_2^H).$$
To give a positive answer to the Modular Isomorphism Problem in this case we should prove that $G\cong H$, or equivalently that $u_i^G=u_i^H$ for $i=1,2$.
Unfortunately, we are only able to prove the statement of \Cref{theorem2}, namely that $u_2^G\equiv u_2^H \mod p$ and, under some extra assumptions, that $u_1^G\equiv u_1^H \mod p$.

Since the Modular Isomophism Problem has positive solutions for metacyclic groups \cite{San96}, and for $2$-generated groups of class $2$ \cite{BdR20}, we may assume that the groups $G$ and $H$ are not metacyclic, and both are of class at least 3.
The first is equivalent to $\max(o_1,o_2)>0$ and the second is equivalent to
$\max(o'_1,o'_2)<m$. 
In particular, $m\ge 2$. Moreover $n_2\geq 2$, as otherwise $n_2<m$ and  condition \ref{4} yields $1=n_2=2m-o_1-o_2'$, but this last quantity is strictly greater than $1$ because $\max(o_1,o_2')<m$, by condition \ref{2} and since $\Ge$ is not metacyclic. We also have that $o_1\neq o_2$ by condition \ref{3}. Finally, if $o_i'=0$ for some $i\in \{1,2\}$, then $u_i^G=1=u_i^H$ by conditions \ref{5} and \ref{6}; therefore we can assume that $\max(o_1',o_2')>0$.
Thus the conditions in \eqref{Assumptions} hold, so we can freely use the statements of the previous section.

In order to deal with $G$ and $H$ simultaneously, in the remainder of the paper $\GG$ denotes a 2-generated finite $p$-group with cyclic derived subgroup such that  
	$$\inv(\GG)=(p,m,n_1,n_2,o_1,o_2,o'_1,o'_2,u_1^\GG,u_2^\GG).$$ 
	
\subsection{Proof of \Cref{theorem2}}
	
Recall that $o=\max(o_1,o_2)$. We let 
	$$N_\GG = \begin{cases}
	\Omega_{m-o-1}(\GG:Z(\GG)\GG'), & \text{if either } o_1=0 \text{ or } o_2=0 \text{ and } o'_1\ge o'_2;\\
	\Omega_{n_2-1}(\GG:\GG'), &  \text{otherwise}
	\end{cases}$$
and 
$$\mathcal{N}_\GG=\frac{ \augNor{N_\GG}{\GG}\cap \aug{\GG}^p}{\aug{N_\GG} \aug{\GG}}.$$
By \Cref{lemma:canonicalsubgroups}, the subquotients $\augNor{N_\GG}{\GG}$, $\Aug{n}{N_\GG,\GG}$ and $\mathcal{N}_\GG$ are canonical. 
Moreover, 
\begin{equation}\label{NGG}
N_\GG=\GEN{a,d,e}, \quad \text{ where } \quad 
(d,e) = \begin{cases}
(b_1^p,b_2^{p^{o_2+1}}), & \text{if } o_1=0; \\ 
(b_2^p,b_1^{p^{o_1+1}}), &  \text{if } o_2=0 \text{ and } o'_1\ge o'_2; \\
(b_2^p,b_1^{p^{n_1-n_2+1}}), & \text{otherwise};
\end{cases}
\end{equation}
and $\mathcal{N}_\GG$ is spanned by the classes of $d-1$ and $e-1$.

\begin{lemma}\label{JnN}
For every $n\ge 0$, $\Aug{n}{N_\GG,\GG}=\aug{N_\GG}^n\aug{\GG}$.
\end{lemma}

\begin{proof}
Suppose first that either $o_1=0$ or $o_2=0$ and $o'_1\ge o'_2$. Then $\gamma_1^\GG(N_\GG)=\GG$,  $\gamma_2^\GG(N_\GG)= (\GG')^p$, and $\gamma_i^\GG(N_\GG)=1$ for $i\geq 3$.
Since $\GG'\subseteq N_\GG$ and , it follows that
$$\aug{N_\GG}^{n-1} \augNor{(\GG')^p}{\GG} \subseteq
\aug{N_\GG}^{n-1+p} k\GG \subseteq  \aug{N_\GG}^{n}\aug{\GG}.$$
Then the desired equality follows from \eqref{eq72}.

Suppose that $o_1\ne 0$ and either $o_2\ne 0$ or $o'_1<o'_2$. Then again $\gamma_1^\GG(N_\GG)=\GG$,  $\gamma_2^\GG(N_\GG)= (\GG')^p$ and $\aug{N_\GG}^{n-1} \augNor{(\GG')^p}{\GG} \subseteq \aug{N_\GG}^{n}\aug{\GG}$. For $i\ge 3$, an easy induction argument, using the description of $N_\GG$ in \eqref{NGG}, shows that $\gamma_i^\GG(N_\GG) = (\GG')^{p^{1+(i-2)k}}$, where $k=n_1-n_2+1+m-o_1$ if $o_2=0$, and $k=1+m-o_2$ otherwise.  
Either way $k\ge 2$ and hence
	$$\aug{N_\GG}^{n+1-i} \augNor{\gamma_i^\GG(N_\GG)}{\GG} \subseteq 
	\aug{N_\GG}^{n+1-i} \augNor{(\GG')^{p^{1+(i-2)k}}}{\GG} \subseteq
	\aug{N_\GG}^{n+1-i+p^{1+(i-2)k}}k\GG \subseteq \aug{N_\GG}^n\aug{\GG}.$$
Then again \eqref{eq72} yields the desired equality.
\end{proof}

Denote 
	$$\ell=\begin{cases} n_1+o'_1-2, & \text{if } o_1=0; \\
	n_2+o'_2-2, & \text{otherwise}.\end{cases}$$ 
Combining \Cref{JnN} and \eqref{Usefi} and using regularity it is easy to obtain 
\begin{equation}\label{JUsefi}
\GG\cap (1+\Aug{p^{\ell}}{N_\GG,\GG}) = 1.
\end{equation}

The next lemma covers most cases of \Cref{theorem2}.

\begin{lemma}\label{MIPUesCaso1}
The following hold:
\begin{enumerate}
	\item $u_2^G\equiv u_2^H \mod p$.
	\item If $o_1o_2=0$ then $u_1^G\equiv u_1^H \mod p$.
\end{enumerate}
\end{lemma}

\begin{proof}
Let $t\in \{1,2\}$ with $t=2$ in case $o_1o_2\ne 0$, and let $s$ be the other element of $\{1,2\}$, i.e. $\{s,t\}=\{1,2\}$.
We have to prove that $u_t^G\equiv u_t^H\mod p$.
If $a_t=0$ then $u_t^G=u_t^G=1$, so we assume that $a_t\ne 0$. 
In particular, $o'_t>0$ and $o_s>0$.	
Therefore
	$$t=\begin{cases} 
	1, & \text{if } o_1=0;\\
	2, & \text{otherwise}.
	\end{cases}$$	
So, $\ell=n_t+o'_t-2$. If $t=1$ then $n_1+o'_1+o_2>n_2+o'_2$, by condition \ref{5}, as $a_1>0$. 
If $t=2$ and $o_1'\ge o_2'$ then, by condition \ref{6}, $o'_2-o'_1\le 0<a_2\le o'_2-o'_1+\max(0,o_1+n_2-n_1)$ and hence  $n_1+o_1'< n_2+o_2'+o_1$ and $o_2=0$. 

We claim that for $x,y\in k$ 
\begin{equation}\label{Lamdaell}
\Lambda^{\ell}_{N_\GG}(x(d-1)+y(e-1)+\aug{N_\GG}\aug{\GG}) = x u_t^\GG (a^{p^{m-1}}-1)+\Aug{p^{\ell}}{N_\GG,\GG}.
\end{equation}
Indeed, if $t=1$ then $o_1=0$, $o'_1>0$, $n_1+o'_1+o_2>n_2+o'_2$, $\ell=n_1+o'_1-2$, $d=b_1^p$ and $e=b_2^{p^{o_2+1}}$. Thus 
\begin{eqnarray*}
\Lambda^{\ell}_{N_\GG}(x(d-1)+y(e-1)+\aug{N_\GG}\aug{\GG}) &=& 
x(b_1^{p^{n_1+o'_1-1}}-1) + y(b_2^{p^{n_1+o'_1+o_2-1}}-1)+\Aug{p^{\ell}}{N_\GG,\GG}	\\
&=& xu_1^\GG(a^{p^{m-1}-1}-1)+\Aug{p^{\ell}}{N_\GG,\GG}.
\end{eqnarray*}
Suppose that $t=2$. Then $o'_2>0$, $o_1>0$ and $\ell=n_2+o'_2-2$. 
If $o'_2\le o'_1$ then $o_2=0$ and $n_2+o'_2+o_1> n_1+o'_1$, and \eqref{Lamdaell} follows as in the previous case.
If $o'_2>o'_1$ then 
\begin{eqnarray*}
\Lambda^{\ell}_{N_\GG}(x(d-1)+y(e-1)+\aug{N_\GG}\aug{\GG}) &=& 
x(b_2^{p^{n_2+o'_2-1}}-1) + y(b_1^{p^{n_1+o'_2-1}}-1)+\Aug{p^{\ell}}{N_\GG,\GG}	\\
&=& xu_2^\GG(a^{p^{m-1}-1}-1)+\Aug{p^{\ell}}{N_\GG,\GG}.
\end{eqnarray*}
This finishes the proof of \eqref{Lamdaell}.
	
By \eqref{JUsefi} and \eqref{Lamdaell}, $\Lambda^{\ell}_{N_\GG}(\mathcal{N}_\GG)$ is one dimensional spanned by the class of $a^{p^{m-1}}-1$. 
Moreover, as $o'_t>0$, $a^{u_t^\GG p^{m-1}} = d^{p^\ell} \in N_\GG^{p^{\ell}}$ and hence the natural projection defines an isomorphism 
$$\Delta'_\GG:\frac{\aug{\GG'}^{p^{m-1}} k\GG}{\aug{\GG'}^{p^{m-1}}\aug{\GG}}\to
\Lambda_{N_\GG}^\ell(\mathcal{N}_\GG).$$

Using \eqref{NGG} and \eqref{ImagenLambda} it is easy to see that the natural projections 
	$$\eta_\GG:\mathcal N_\GG\to  \frac{\augNor{N_\GG}{\GG}+\aug{\GG}^{p+1}} { \augNor{\GG'}{\GG}+\aug{\GG}^{p+1}} \qand  
	\Lambda_\GG^1(\mathcal C_\GG)  \to \frac{\augNor{N_\GG}{\GG}+\aug{\GG}^{p+1}} { \augNor{\GG'}{\GG}+\aug{\GG}^{p+1}}$$
make sense, their images coincide and the second map is injective. Thus the natural projection induces an isomorphism $\Lambda_\GG^1(\mathcal C_\GG)\to \eta_\GG(\mathcal N_\GG)$. 
On the other hand, by \eqref{LambdaTildeIso}, $\tilde\Lambda^1_\GG:\mathcal{C}_\GG \to \Lambda^1_\GG(\mathcal{C}_\GG)$ is an isomorphism.
Composing these isomorphisms we obtain an isomorphism 
	$$\hat\Lambda^1_\GG:\mathcal{C}_\GG\to \Imagen(\eta_\GG), \quad w+\aug{\GG}^2\mapsto w^p+\augNor{\GG'}{\GG}+\aug{\GG}^{p+1}.$$
This provides another canonical map 
$$\nu_\GG = (\hat\Lambda^1_\GG)^{-1}\circ \eta_\GG: \mathcal{N}_\GG \to \mathcal{C}_\GG,\quad  w+\aug{N_\Gamma}\aug{\Gamma}\mapsto (\hat\Lambda^1_\GG)^{-1}(w+\augNor{\GG'}{\GG}+\aug{\GG}^{p+1}).$$

Define the linear map
	$$\mu_\GG:\mathcal{C}_\GG \to \frac{\aug{\GG'}^{p^{m-1}} k\GG}{\aug{\GG'}^{p^{m-1}}\aug{\GG}}$$
sending the class of $x(b_t-1)$ to the class of $xu_t^\GG(a^{p^{m-1}}-1)$.
	A straightforward calculation shows that the following diagram commutes. 
$$\xymatrix{
	\mathcal{N}_\GG \ar[d]_-{\nu_\GG}
	\ar[rr]^-{(\Delta'_\GG)^{-1}\circ\Lambda_{N_\GG}^{p^{\ell}}}& &
	\frac{\aug{\GG'}^{p^{m-1}} k\GG}{\aug{\GG'}^{p^{m-1}}\aug{\GG}}   \\
	\mathcal{C}_\GG \ar[rru]_-{\mu_\GG}   &   &
} $$
As the vertical map is surjective, $\mu_\GG$ is the unique map making the previous commutative. Then $\mu_\GG$ is canonical, since the other maps in the diagram are so.
	
	Consider the following equation where $X$ stands for an element of $k$.
	\begin{equation}\label{diagram1}
	X\cdot \left(\Lambda_{\GG'}^{p^{m-1}}\circ \Delta_\GG^{-1}\circ\zeta_\GG^1\right) =
	\mu_\Gamma\circ (\tilde \Lambda_{\GG }^{p^ {m-o}})^{-1}\circ\pi_\GG^{-1}\circ \zeta_\GG^2.
	\end{equation}
	Here, given a map $f$ with codomain in a vector space over $k$ and $x\in k$,  $x\cdot f$ denotes the map given by $(x\cdot f)(w)=xf(w)$, for each $w$ in the domain of $f$.
	The unique solution for equation \eqref{diagram1} is $X=\delta u_t^\GG 1_k$.
	Since all the maps involved are canonical, the solution when $\GG=G$
	coincides with the solution when $\GG=H$. Furthermore, $p\nmid \delta$ and thus $u_t^G \equiv u^H_t \bmod p$, as desired.
\end{proof}

Most of the remaining cases of \Cref{theorem2} are covered by the next lemma. 
 
\begin{lemma}\label{MIPUeso1o2>0}
If $n_1+o_1'\neq n_2+o_2'$, then $u_1^G\equiv u_1^H\mod p$.  
\end{lemma}

\begin{proof} 
By \Cref{MIPUesCaso1} we may assume that $o_1o_2\ne 0$. 
Hence condition \ref{3} and the hypothesis imply $n_1+o_1'> n_2+o_2'$.
As in the proof of \Cref{MIPUesCaso1} we may assume that $a_1>0$ and hence $o_1'>0$. Consider the subgroup 
	 $$M_\GG=\Omega_{n_2-m+o_1}(\GG:\GG')=\GEN{b_1^{p^{n_1-n_2+m-o_1}}, b_2^{p^{m-o_1}}, a}.$$
Recall that $c=b_1^{-\delta p^{m-o_2}} b_2^{\delta p^{m-o_1}}a$ and $\frac{\ZZ(\aug{\GG})+ \aug{\GG}^{p^m}}{\aug{\GG}^{p^m}}$ is spanned by the classes of $c-1,(c-1)^2,\dots,(c-1)^{\frac{p^m-1}{2}}$. 
The natural projection
	 $$\zeta_\GG^3: \frac{\ZZ(\aug{\GG})+ \aug{\GG}^{p^m}}{\aug{\GG}^{p^m}} \to \frac{\ZZ(\aug{\GG})+ \aug{\GG}^{p^{m-o_2}+1}+\augNor{M_\GG}{\GG}}{\aug{\GG}^{p^{m-o_2}+1}+\augNor{M_\GG}{\GG}}$$
maps the class of $x(c-1)+ y(c-1)^2+ \dots$ to the class of $-x\delta (b_1^{p^{m-o_2}}-1)$, which is non-zero if $x\neq 0$ because $n_1-n_2+m-o_1>m-o_2$. 
So $\Imagen(\zeta_\GG^3)$ is $1$-dimensional. 

Now consider the composition 
$$\hat{\Lambda}^{m-o_2}_\GG: \frac{\aug{\GG}}{\aug{\GG}^2} \stackrel{\Lambda^{m-o_2}_\GG}{\longrightarrow} 
\frac{\aug{\GG}^{p^{m-o_2}}}{\aug{\GG}^{p^{m-o_2}+1}} \longrightarrow 
\frac{\aug{\GG}^{p^{m-o_2}}+\augNor{M_\GG}{\GG}}{\aug{\GG}^{p^{m-o_2}+1}+\augNor{M_\GG}{\GG}}$$
where the second map is the natural projection. 
It maps $x(b_1-1)+y(b_2-1)$ to $x(b_1^{p^{m-o_2}}-1)$, so $\Imagen(\hat{\Lambda}^{m-o_1}_\GG)=\Imagen(\zeta^3_\GG)$. 

The image of $\Lambda_\GG^{n_1+o'_1-1}$ is the subspace of $\aug{\GG}^{p^{n_1+o'_1-1}}/\aug{\GG}^{p^{n_1+o'_1-1}+1}$ spanned by the class of $a^{p^{m-1}}-1$. It coincides with the image of the natural projection 
$$\frac{\aug{\GG'}^{p^{m-1}} k\GG}{\aug{\GG'}^{p^{m-1}}\aug{\GG}} \to \frac{\aug{\GG}^{p^{n_1+o'_1-1}}}{\aug{\GG}^{p^{n_1+o'_1-1}+1}}.$$
Thus this natural projection yields an isomorphism $\tilde\Delta_\GG:\frac{\aug{\GG'}^{p^{m-1}} k\GG}{\aug{\GG'}^{p^{m-1}}\aug{\GG}}\to \Imagen(\Lambda_\GG^{n_1+o'_1-1})$.

Let $\mu_\GG:\Imagen(\zeta^3_\GG)\to \frac{\aug{\GG'}^{p^{m-1}} k\GG}{\aug{\GG'}^{p^{m-1}}\aug{\GG}}$ be the map that sends the class of $x (b_1^{p^{m-o_1}}-1)$ to the class of $xu_1(a^{p^{m-1}}-1)$.
Then it is easy to see that the following diagram commutes
$$  \xymatrix{ 
 	\frac{\aug{\GG}}{\aug{\GG}^2} 
 	\ar[rr]^-{  \tilde \Delta_\GG^{-1}\circ\Lambda_{\GG }^{ n_1+o_1'-1 }}   \ar[d]_-{\hat{\Lambda}^{m-o_2}_\GG}& &  
 	\frac{\aug{\GG'}^{p^{m-1}} k\GG}{\aug{\GG'}^{p^{m-1}}\aug{\GG}}   \\
 	\Imagen( \zeta_\GG^3) \ar[rru]_-{  \mu_\GG}   & &
 }$$
As the vertical map is surjective, $\mu_\GG$ is the unique map making the previous commutative, so $\mu_\GG$ is canonical. 
Then $-\delta u_1^{\Gamma} 1_k$ is the unique solution of the equation 
$$X \cdot (\Lambda_{\GG'}^{p^{m-1}} \circ \Delta_\GG^{-1}\circ\zeta_\GG^1)= \mu_\GG\circ \zeta_\GG^3.$$
Arguing as at the end of the proof of \Cref{MIPUesCaso1} we conclude that $u_1^G\equiv u_1^H\mod p$.
\end{proof}

The proof of \Cref{MIPUeso1o2>0} fails if $n_1+o_1'=n_2+o_2'$, because in that case  $\ker(\hat{\Lambda}^{m-o_2}_\GG)\not\subseteq \ker(\Delta_\GG^{-1}\circ \Lambda_{\GG}^{n_1+o_1'-1})$, and hence there is no map $\mu_\GG$ such that 
$\mu_\GG \circ \hat{\Lambda}^{m-o_2}_\GG = \Delta_\GG^{-1}\circ \Lambda_{\GG}^{n_1+o_1'-1}$. 
However, some special subcases can  be handled with slight modifications of the previous arguments.

For a non-negative integer $n$ define the map 
\begin{eqnarray*}
\Upsilon^n_\GG:\frac{\ZZ(\aug{\GG})+\aug{\GG}^{p^m}}{\aug{\GG}^{p^m}} & \longrightarrow &
\frac{\ZZ(\aug{\GG})+\aug{\GG}^{p^{n+m}}+\aug{\GG'}^{p^{m-1}}\aug{\GG}}{\aug{\GG}^{p^{n+m}}+\aug{\GG'}^{p^{m-1}}\aug{\GG}} \\
w+\aug{\GG}^{p^m} & \mapsto & w^{p^n}+\aug{\GG}^{p^{n+m}}+\aug{\GG'}^{p^{m-1}}\aug{\GG}.	
\end{eqnarray*} 
It is well defined because the elements of $\ZZ(\aug{\GG})$ are central.
 
 \begin{lemma}\label{MIPUesSpecial}
If $o_1o_2>0$, $n_1+o_1'=n_2+o_2'=2m-o_1$ and $u_2^G\equiv u_2^H \equiv 1 \mod p^{o_1+1-o_2}$, then $u_1^G\equiv u_1^H \mod p$.
 \end{lemma} 

\begin{proof}
As in previous proofs we may assume that $a_1\ne 0$ and hence $0<o'_1$.
As $o_1o_2>0$ implies $n_1>n_2$, necessarily $1\le o_1'<o_2'$.
Recall that
	$ \ZZ(\GG)=\GEN{b_1^{p^m},b_2^{p^m},c}   $,	where  $c =b_1^{-\delta p^{m-o_2}}b_2^{\delta p^{m-o_1}}a$.

We claim that
\begin{equation}\label{Congruenciadelta}
(\delta u_2^\GG+1)p^{m+o_2-o_1-1} \equiv 0 \bmod p^m.
\end{equation}
To prove this, it suffices to show that $\delta\equiv -1 \mod p^{o_1+1-o_2}$.
As $v_p(r_2-1)=m-o_2$, $m-o_1\ge 1 = m+1-o_2-v_p(r_2)$. 
Hence \cite[Lemma~A.2]{OsnelDiegoAngel} yields  $\Ese{r_2}{\delta p^{m-o_1}}\equiv \delta p^{m-o_1} \mod p^{m+1-o_2}$. Thus \eqref{eq:CenteredCongruence} implies that $\delta \equiv -1 \mod p^{o_1+1-o_2}$. This proves \eqref{Congruenciadelta}.
	
Next we claim that
\begin{equation}\label{powerofc}
c^{p^{ n_1+o_1'-1 -m+o_2}}=a^{-\delta u_1^\GG p^{m-1}}.
\end{equation}
Indeed, first observe that condition  implies
	\begin{equation}\label{Desigualdad}
	n_1+o_1'-1=2m-o_1-1 \ge 2m -o_2 +n_2-n_1 = 2m-o_2-o_2'+o_1' \geq 2m-o_2-o_2' \ge
		m-o_2.
	\end{equation}
Thus the exponent in the left side of \eqref{powerofc} is a positive integer.
Observe that
	$$\GEN{b_1^{p^{m-o_2}}, b_2^{p^{m-o_1}},a} $$
is a regular group with derived subgroup  $\GEN{a^{p^{2m-o_1-o_2}}}$.
As $m-o_2'+2m-o_1-o_2= 3m-o_1-o_2-o_2'\geq 2m -o_1>m $ (since $o_2+o_2'\leq m$), we derive that
	$$c^{p^{m-o_2'}}= b_1^{-\delta p^{2m-o_2-o_2'}} b_2^{\delta p^{2m-o_1-o_2'}} a^{p^{m-o_2'}}= b_1^{-\delta p^{2m-o_2-o_2'}} b_2^{\delta p^{n_2}} a^{p^{m-o_2'}}=
	b_1^{-\delta p^{2m-o_2-o_2'}} a^{(\delta  u_2^\GG +1)p^{m-o_2'}}.$$
As $b_1^{p^{2m-o_2-o_2'}}\in \ZZ(\GG)$ and recalling \eqref{Desigualdad} we get
\begin{align*}
c^{p^{ n_1+o_1'-1 -m+o_2}} &=	(c^{p^{m-o_2'}})^{p^{n_1+o_1'-1-(2m-o_2-o_2')}}= b_1^{-\delta p^{n_1+o_1'-1}}  a^{ (\delta u_2^\GG+1) p^{n_1+o_1'-1-m+o_2}} \\
	&=a^{-\delta u_1^\GG p^{m-1}}a^{(\delta u_2^\GG+1) p^{m+o_2-o_1-1}} =
	a^{-\delta u_1^\GG p^{m-1}},
\end{align*}
where the last equality follows from \eqref{Congruenciadelta}. This proves \eqref{powerofc}.

Using \eqref{powerofc} we obtain that $\Upsilon^{n_1+o'_1+o_2-m-1}_\GG$ maps the class of $\sum_{i=1}^{\frac{p-1}{2}} x_i(c-1)^i$, with $x_i\in k$, to the class of $-x_1\delta u_1^\GG (a^{p^{m-1}}-1)$. If $x_1\ne 0$, then the latter is not the class zero, by \Cref{UsefiFrattini}. Then the natural projection defines an isomorphism
$\pi_\GG:\frac{\aug{\GG'}^{p^{m-1}} k\GG }{\aug{\GG'}^{p^{m-1}}\aug{\GG}}\to \Imagen(\Upsilon^{n_1+o'_1+o_2-m-1}_\GG)$.
So we have a canonical map
	$$\pi_\GG^{-1}\circ \Upsilon^{n_1+o'_1+o_2-m-1}_\GG:\frac{\ZZ(\aug{\GG})+\aug{\GG}^{p^m}}{\aug{\GG}^{p^m}} \to \frac{\aug{\GG'}^{p^{m-1}} k\GG }{\aug{\GG'}^{p^{m-1}}\aug{\GG}},
	$$
mapping the class of $\sum_{i=1}^{\frac{p-1}{2}} x_i(c-1)^i$ to the class of $x_1(-\delta u_1^\GG) (a^{p^{m-1}}-1)$. But we also have the canonical map
	$$\Lambda_{\GG'}^{m-1} \circ \Delta_\GG^{-1}\circ \zeta_\GG^1: \frac{\ZZ(\aug{\GG})+\aug{\GG}^{p^m}}{\aug{\GG}^{p^m}} \to \frac{\aug{\GG'}^{p^{m-1}}k\GG}{\aug{\GG'}^{p^{m-1}}\aug{\GG}}$$
that maps the class of $\sum_{i=1}^{\frac{p-1}{2}} x_i(c-1)^i$ to the class of $x_1(a^{p^{m-1}}-1)$.
Thus the unique element $x\in k$ such that $\pi_\GG^{-1}\circ \Upsilon^{n_1+o'_1+o_2-m-1}_\GG=x\cdot (\Lambda_{\GG'}^{m-1} \circ \Delta_\GG^{-1}\circ \zeta_\GG^1)$
is $-\delta u_1^\Gamma 1_k$. Since all the maps are canonical, this has to be the same for $\GG=G$ and $\GG=H$. Hence $u_1^G\equiv u_1^H \mod p$.
\end{proof}

\Cref{theorem2} follows at once from Lemmas~\ref{MIPUesCaso1}, \ref{MIPUeso1o2>0} and \ref{MIPUesSpecial}.

\subsection{Proof of \Cref{theorem1}}\label{SubsectionTheoremA}

Since $\psi(\augNor{G'}{G})=\augNor{H'}{H}$, we have that $\psi(\augNor{(G')^{p^n}}{G})=\augNor{(H')^{p^n}}{H}$ for each $n\geq 1$. Hence $\psi$ induces isomorphims $\psi_n: k(G/(G')^{p^n})\to k(H/(H')^{p^n})$.

\medskip
	
We first proof \Cref{theorem1}\eqref{theorem1.1}. 
By \eqref{gamma}, $\gamma_3(G)=(G')^{p^{m-\max(o_1,o_2)}}$. 
Hence, $\psi_{m-\max(o_1,o_2)+1}$ is an isomorphism $k(G/\gamma_3(G)^p)\cong k(H/\gamma_3(H)^p)$. Hence we can assume that $|\gamma_3(G)|=|\gamma_3(H)|=p$, so necessarily $\max(o_1,o_2)=1$. 
This means that $\{o_1,o_2\}=\{0,1\}$, by condition \ref{3}.
Thus  $a_1\leq o_2$ and $a_2\leq o_1$.
Then $1\leq u_i^\GG <p$ for $i\in\{1,2\}$ and $\GG\in \{G,H\}$ by conditions \ref{5} and \ref{6}. Therefore  $u_1^G=u_1^H$ and $u_2^G=u_2^H$ by \Cref{theorem2}, and the result follows.
This proves \Cref{theorem1}\eqref{theorem1.1}.

\medskip

To prove \Cref{theorem1}\eqref{theorem1.2} we need one more result,  which allows us to recover $u_i^\GG$ modulo a higher power of $p$ in very special situations (see \Cref{MIPUesHigherPower}). For that, we define
	$$q^\GG=\min\{n\geq 0 : \Omega_1(\GG') \cap \M_{p^n}(\GG)=1\}.$$
We claim that 
\begin{equation}\label{lemma:descriptionq}
	q^\GG=\begin{cases}
	m, & \text{if } o'_1=o'_2=0; \\
	n_2+o_2',& \text{if }0=o_1'<o'_2; \\
\max(n_1+o_1',n_2+o_2'), &\text{if } o'_1>0. 
	\end{cases}
\end{equation}
Indeed, first recall that $n_1\geq m $ by condition \ref{4}. Moreover, $n_2+o_2'\geq m$, since otherwise, by the same condition, $n_2=2m-o_1-o_2'$, so $m< 2m-o_1 =n_2+o_2'\leq m$, a contradiction. Clearly, $m\le q^\GG$, since $1\ne a^{p^{m-1}}\in \Omega_1(\GG')\cap \M_{p^{m-1}}(\GG)$. Moreover, using regularity and \eqref{Mpn} we derive that if $n\ge m$, then $\M_{p^n}(\GG)=\GEN{b_1^{p^n},b_2^{p^n}}$. 
If $o'_1=o'_2=0$, then $\M_{p^m}(\GG)\cap \Omega_1(\GG')=1$, so $q^\GG=m$. 	
Suppose that $0=o_1'<o_2'$. Then $a^{p^{m-1}}\in \M_{p^{n_2+o_2'-1}}(\GG)$, but $\M_{p^{n_2+o_2'}}(\GG)=\GEN{b_1^{p^{n_2+o_2'}}}$, which does not intersect with $\GG'$.
Thus $q^\GG=n_2+o'_2$.
Finally suppose that $o_1'>0$.  Then $a^{p^{m-1}}\in \M_{p^{\max(n_1+o_1',n_2+o_2')-1}}(\GG)$ because if $n_2+o_2'> n_1+o_1'$ then $o_2'>0$ since $n_1\ge n_2$.
As $\M_{p^{\max(n_1+o_1',n_2+o_2')}}(\GG)=1$, we conclude that $q^\Gamma=\max(n_1+o'_1,n_2+o'_2)$.
This finishes the proof of \eqref{lemma:descriptionq}.

\begin{lemma}\label{MIPUesHigherPower} Let $t$ be a positive integer such that $t\leq 2m-1-q^G$. 
	\begin{enumerate}
		\item \label{MIPUesHigherPower1}
		Suppose that $o_1=0$ and $n_1=2m-o_2-o_1'$. 
		If $u_1^G\equiv u_1^H \equiv -1 \mod p^{t}$, then  $u_1^G\equiv u_1^H \mod p^{t +1}$.
		\item \label{MIPUesHigherPower2}
		
		Suppose that $o_2=0$ and $n_2=2m-o_1-o_2'$. If $u_2^G\equiv u_2^H \equiv 1 \mod p^{t}$, then  $u_2^G\equiv u_2^H\mod p^{t+1}$.
	\end{enumerate} 
\end{lemma}

\begin{proof}
Suppose first that the hypotheses of \eqref{MIPUesHigherPower1} hold.  
If $a_1\le t$ then $u_1^G=u_2^H=-1+p^{a_1}$. Thus we may assume that $t<a_1$ and in  particular $t<o'_1$. Then $q^\GG=\max(n_1+o'_1,n_2+o'_2)$. 
Write $u_1^\GG =-1+ v_1^\GG p^{t}$.  
Recall that $\ZZ(\GG)= \GEN{b_1^{p^m}, b_2^{p^m}, c= b_1^{p^{m-o_2}}a}$, by \eqref{eq:centerGens}.
As $o_1=0$, $[b_1,a]=1$ and  hence
	$$(b_1^{p^{m-o_2}}a)^{p^{m-o_1'}}=b_1^{p^{n_1}}a^{p^{m-o_1'}}=a^{(u_1^\GG+1) p^{m-o_1'}}=a^{v_1^\GG p^{m-o_1'+t}}.$$
Therefore
	$$(b_1^{p^{m-o_2}}a)^{p^{m-t-1}}=((b_1^{ p^{m-o_2}}a)^{p^{m-o_1'}})^{p^{o_1'-t-1}}= a^{v_1^\GG p^{m-1}}. $$
Then $\Upsilon^{m-t-1}_\GG$ maps the class of $x(c-1)+y(c-1)^2+ \dots  $ to the class of $x v_1^\GG  (a^{p^{m-1}}-1) $. 
Observe that $a^{p^{m-1}}\not\in \M_{2m-t-1}(\GG)$ since $2m-t-1\geq q^\GG$. 
Hence $(a^{p^{m-1}}-1)\not\in \aug{\GG}^{p^{2m-t-1}}+\aug{\GG'}^{p^{m-1}}\aug{\GG}$, by  \Cref{UsefiFrattini}.  
Thus $\Imagen(\Upsilon^{m-t-1}_\GG)$ has dimension $1$, and the natural projection
	$$\omega_\GG: \frac{\aug{\GG'}^{p^{m-1}} k\GG}{\aug{\GG'}^{p^{m-1}}\aug{\GG}} \to \Imagen(\Upsilon^{m-t-1}_\GG)$$ is an isomorphism. If $x\in k$, then
	$$ (\omega_\GG)^{-1}\circ \Upsilon^{m-t-1}_\GG = x\cdot (\Lambda_{\GG'}^{m-1} \circ  \Delta_\GG^{-1}\circ  \zeta_\GG^1)$$
	if and only if $x=  v_1^\GG\cdot 1_k $.
	As this holds both for $\GG=G$ and for $\GG=H$ and all the maps are canonical, we conclude that $v_1^G\equiv v_1^H \mod p$, so $u_1^G\equiv u_1^H\mod p^{t+1}$.
This finishes the proof of \eqref{MIPUesHigherPower1}.	 
	
Under the assumptions of \eqref{MIPUesHigherPower2}, the congruence in \eqref{eq:CenteredCongruence} yields $\delta\equiv -1 \mod p^{o_1}$,   and hence 
	$\ZZ(\GG)=\GEN{b_1^{p^m}, b_2^{p^m}, c=b_2^{-p^{m-o_1}} a}$.
Then setting $u_2^\GG=1+v_2^\GG p^t$ and arguing as above we obtain
	$(b_2^{-p^{m-o_1}}a)^{p^{m-t^\GG-1}}= a^{-v_2^\GG p^{m-1}}$. 
The rest of the proof is completely analogous to the previous case. 
\end{proof}
 
 Observe that \Cref{GG'p2} is equivalent to the following lemma.

\begin{lemma}\label{lemma:n2}
If $n_2\leq 2$, then $G\cong H$.
\end{lemma}

\begin{proof}
Recall that we are assuming that \eqref{Assumptions} holds, so $m\ge 2$ and we may assume that $n_2=2$. 
If $m=2$ then $|\gamma_3(\GG)|=p$, and hence the result follows from \Cref{theorem1}\eqref{theorem1.1}. 
Thus we assume $m\geq 3$. 
Then $n_2<m$, and by  condition \ref{4}, $2=n_2=2m-o_1-o_2'$ and $u_2^\GG \equiv 1 \mod p^{m-2}$. Then $2(m-1)=o_1+o_2'$. 
Since $o_1<m$ by condition \ref{2}, and $o_2'<m$ by \eqref{Assumptions}, we derive that $o_1=o_2'=m-1$. 
As $o_i+o_i'\leq m$ by condition \ref{2}, also $o_1'\leq 1$ and $o_2\leq 1$.
Therefore $1\leq u_1^\GG \leq p$.
Then \Cref{theorem2} implies that $u_1^G=u_1^H$ or condition \eqref{theorem2.2} in the theorem holds. In the latter case $o_1o_2>0$  and $n_1+o_1'=n_2+o_2'=m+1$.
The former implies $n_1>m$ by \eqref{eq:n1=m}. Therefore  $o_1'=0$, so $u_1^G=1=u_1^H$.

Observe that $1\leq u_2^\GG \leq p^{a_2}$, for otherwise, $o_2=1$ and $n_1+o'_1-m-1 = n_1-n_2+o'_1-o'_2=0<a_1\le o'_1\le 1$ by condition \ref{6}, so $n_1=m$ and $o_1o_2>0$, in contradiction with \eqref{eq:n1=m}.
If $o_2=1$, then $a_2\leq o_1-o_2=m-2$, so $1\leq u_2^\GG \leq p^{m-2}$ and hence $u_2^G=u_2^H$. Thus we assume $o_2=0$.
Suppose that $o_1'=0$.
Then by \eqref{lemma:descriptionq} $q^\GG=n_2+o_2'=m+1$.
Thus $m-2\leq 2m-1-q^\GG =m-2$.
Therefore \Cref{MIPUesHigherPower}\eqref{MIPUesHigherPower2} with $t=m-2$ yields that $u_2^G\equiv u_2^H\mod p^{m-1}$, i.e., $u_2^G=u_2^H$.
Now suppose that $o_1'=1$.
Since $n_1\geq m $ by condition \ref{4}, $q^\GG=n_1+1$.
If $n_1>m$, then by condition \ref{6} $a_2=o_2'-o_ 1+\max(0,m+1-n_1)=m-2$ and $1\leq u_2^\GG\leq p^{m-2}$, so $u_2^G=u_2^H$. Hence we assume $n_1=m$. Then $q^\GG=m+1$ and $m-2\leq 2m-1-q^\GG=m-2$. Thus, again \Cref{MIPUesHigherPower}\eqref{MIPUesHigherPower2} with $t=m-2$ yields $u_2^G\equiv u_2^H\mod p^{m-1}$, i.e., $u_2^G=u_2^H$.
\end{proof}

We are finally ready to prove \Cref{theorem1}\eqref{theorem1.2}.
Via the isomorphism $\psi_3$ introduced at the beginning of \Cref{SubsectionTheoremA},  we can assume that $(G')^{p^3}=1=(H')^{p^3}$, i.e., $m\le 3$. 
If $n_2\leq 2$, then the result follows from \Cref{lemma:n2}, so we assume $3\leq n_2$. 
If $|\gamma_3(G)|\leq p$, then the result follows from \Cref{theorem1}\eqref{theorem1.1}. Thus we assume $|\gamma_3(G)|=|\gamma_3(H)|=p^2$, so $m=3$. 
Then $\gamma_3(G)=(G')^{p^{m-\max(o_1,o_2)}}$, by \eqref{gamma}, which implies that $\max(o_1,o_2)=2$. 
By condition \ref{3}, we have three possibilities: $0=o_1<o_2=2$, $0=o_2<o_1=2$ and $1=o_2<o_1=2$.

Suppose that $0=o_1<o_2=2$. Then $u_2^G=1=u_2^H$, by condition \ref{6}. Since $m=3$ and $o_2+o_1'\leq m$ by condition \ref{4}, we have that $o_2'\leq 1$, so $1\leq u_1^\GG \leq p$ for $\GG\in \{G,H\}$. Thus $u_1^G=u_1^H$, by \Cref{theorem2}.

Suppose that $0=o_2<o_1=2$. Then $u_1^G=1=u_1^H$ by condition \ref{5}.  Recall that   $m=3\leq n_2$. Then  $o_2'+2=o_2'+o_1\leq m=3$ by condition \ref{4}, so $o_2'\leq 1$. Hence $1\leq u_2^\GG \leq p$ for $\GG \in\{G; H\}$, by condition \ref{6}. Thus $u_2^G=u_2^H$ by \Cref{theorem2}.
	
Finally suppose that $1=o_2<o_1=2$. 
By condition \ref{2},  $o_1'\leq 1$, and since $n_2\geq m$, by condition \ref{4}, $o_2' \leq 1$.
Then $1\leq u_1^\GG \leq p$. 
Observe that neither condition \eqref{theorem2.2} in \Cref{theorem2} nor condition \ref{6}(b) holds since, by  condition \ref{3}, in any of these cases $1=o_1-o_2<n_1-n_2=o_2'-o_1' \leq 1$, a contradiction. 
Therefore $1\leq u_2^\GG \leq p$ and, by \Cref{theorem2}, we derive that  $u_1^G=u_1^H$ and $u_2^G=u_2^H$.

\section{Applications to groups of small order}\label{SectionApplications}

Recall that $p$ is an odd prime and $k$ is the field with $p$ elements.
We first solve the Modular Isomorphism Problem for our target groups when their order is at most $p^{11}$.

 \begin{proposition}\label{corollary}
Let $G$ be a $2$-generated $p$-group with cyclic derived subgroup such that $|G|\leq  p^{11}$. If $kG\cong kH$ for some group $H$, then $G\cong H$.
\end{proposition}

\begin{proof}
We may assume that $G$ is neither metacyclic nor of class at most $2$. Thus conditions \eqref{Assumptions} are satisfied and hence we can use all the results in previous sections. 
Let $G$ and $H$ be a $2$-generated $p$-groups ($p>2$) with cyclic derived subgroup of order at most $p^{11}$, with $kG\cong kH$ and the usual notation $\inv(\GG)=(p,m,n_1,n_2,o_1,o_2,o_1',o_2',u_1^\GG,u_2^\GG)$ for $\GG\in\{G,H\}$. 
If $m\leq 3$, the result follows from \Cref{theorem1}\eqref{theorem1.1}. Thus we assume $m>3$. Then $n_1\geq m>3$ by condition \ref{4}. We can assume that $n_2\geq 3$ by \Cref{lemma:n2}. Thus  $|\GG|=p^{n_1+n_2+m}=p^{11}$. Therefore $n_2=3$ and $m=n_1=4$. 
As $n_2<m$,    by condition \ref{4} $u_2^\GG \equiv 1 \mod p $ and  $8-o_1-o_2'=2m-o_1-o_2'=n_2=3$, so $o_1+o_2'=5$.  
Since $o_1<m$ and $o_2'<m$ because $\GG$ is not metacyclic, we derive that $\{o_1,o_2'\}=\{3,2\}$.  Then, by \eqref{eq:n1=m}, $o_2=0$. Thus $u_1^G=1=u_1^H$.  It also follows that $a_2\leq 2$, so $1\leq u_2^\GG \leq p^2$. Since $2\leq  o_1 $, by condition \ref{4}, $o_1'\leq 2$. Thus $q^\GG\leq \max(n_1+o_1',n_2+o_2') \leq 6$. Write $t=1$, so $t=1= 2m-1-q^\GG$. Therefore, by \Cref{MIPUesHigherPower}\eqref{MIPUesHigherPower2}, $u_2^G\equiv u_2^H\mod p^2$. Thus $u_2^G=u_2^H$.
\end{proof}

For groups of order $p^{12}$, we can solve the Modular Isomorphism Problem except for $p-2$ families of groups of size $p$ each one:
\begin{proposition}
Let $G$ be a $2$-generated finite $p$ group with cyclic derived subgroup and $|G|\leq p^{12}$. If $kG\cong kH$ for some group $H$, then one of the following holds:\begin{enumerate}
	\item $G\cong H$. 
	\item There exist $i\in \{1,\dots,p-2\}$ and $u_1^G,u_1^H \in \{ i+jp: 0\leq j \leq p-1   \}$ such that 
	\begin{align*}
		\inv(G)&=(p,4,4,4,0,2,2,2,u_1^G,1) \text{ and} \\
		\inv(H)&=(p,4,4,4,0,2,2,2,u_1^H,1).
	\end{align*}
\end{enumerate}
\end{proposition}
\begin{proof}
By \Cref{theorem1} and \Cref{corollary}, we may assume that $|G'|>p^3$ and $|G|=p^{12}$. Moreover we can assume that neither $G$ nor $H$ is metacyclic nor of class at most $2$. 
With the notation of \Cref{theorem2}, $\inv(G)$ equals $\inv(H)$ except the last two entries, $(u_1^\GG, u_2^\GG)$, where $\GG\in\{G,H\}$. Moreover, we can assume $n_2\geq 3$ by \Cref{lemma:n2}. Then either $n_2=3$, $m=4$ and $n_1=5$, or $n_2=m=n_1=4$.

\underline{Suppose  that $m=4$, $n_1=5$ and $n_2=3$.} By  condition \ref{4}, $u_2^\GG \equiv 1 \mod p$ and $3=n_2=2m-o_1-o_2'$, so $5=o_1+o_2'$, and hence $\{o_1,o_2'\}=\{2,3\}$ because $o_1<m$ and $o_2'<m$.

Suppose that $o_2'=3$. Then $o_1=2$,   $o_2\leq m-o_2'=1$ and $o_1' \leq m-o_1=2$. Assume that $o_2=1$. Then $a_2\leq o_1-o_2=1$. If condition \ref{6}(b) does not hold, then \Cref{theorem2} yields $u_2^G=u_2^H$. Thus suppose this condition holds.   Then $u_1^\GG\equiv 1 \bmod p$ and $5+o_1'=n_1+o_1'=n_2+o_2'=6$, so $o_1'=1$.  
Hence $1\leq u_1^\GG \leq p$, and we get $u_1^G=u_1^H=1$.  
Moreover $a_2=\min(o_1-o_2,o_2'-o_1')=1$. Thus $u_2^\GG\in\{1,1+p\}$. Summarizing, after exchanging $G$ and $H$, if necessary,
\begin{align*}
	\inv(G)&=(p,4,5,3,2,1,1,3,1, 1) ; \\
	\inv(H)&= (p,4,5,3,2,1,1,3,1, 1+p) .
\end{align*}
But then a straightforward computation, using \eqref{eq:centerGens}, shows that $\ZZ(G)$ has exponent $p^2$ while the exponent of $\ZZ(H)$ is $p^3$, in contradiction with a result of Ward \cite{Ward} (see \cite[Lemma~2.7]{Pas77}). 
 Now assume $o_2=0$. Then $a_1=0$, so $u_1^G=1=u_1^H$. Observe that $a_2=\min (o_1, 3-o_1') \leq 2$. Moreover $3-o_1'=o_2'-o_1'\leq n_1-n_2=2$, so $1\leq o_1'$. If $o_1'=1$ then $q^\GG= 6$, and setting $t=1= 2m-1-q^\GG $, \Cref{MIPUesHigherPower}\eqref{MIPUesHigherPower2} yields $u_2^G=u_2^H$. 
Otherwise, i.e. if $o_1'\geq 2$, then $a_2\leq 1$, and $u_2^G=u_2^H$ by \Cref{theorem2}.

Now suppose that $o_2'=2$. Then $o_1=3$, $o_2\leq m-o_2'=2$ and $o_1'\leq m-o_1=1$. We claim that $u_1^G=u_1^H$. Indeed, if $o_1'=0$ then $u_1^G=1=u_1^H$, and if $o_1'=1$ then condition \eqref{theorem2.2} of \Cref{theorem2} does not hold, and hence that theorem yields the claim.  
Moreover $a_2\leq o_2'=2$ and if condition \ref{6}(b) holds, then $o_2>0$ and $o_1'\ge a_1>0$ so that $a_2\le 1$. 
Thus $1\leq u_2^\GG \leq 2p< p^2$. 
Observe that $q^\GG =\max(5+o_1',5)\leq 6$. 
Then set $t=1\leq 2m-1-q^\GG $, and \Cref{MIPUesHigherPower}\eqref{MIPUesHigherPower2} yields that $u_2^G=u_2^H$.

\underline{Finally, suppose that $m=n_1=n_2=4$.} By condition \ref{3} we have that $o_1=0$. Then $u_2^G=1=u_2^H$. Moreover $a_1=\min(o_1',o_2+o_1'-o_2')$. If $a_1\leq 1$, then $u_1^G=u_1^H$ by \Cref{theorem2}. Thus we assume $a_1\geq 2$,  i.e., $2\leq o_1'\leq 3$ and $2\leq o_1'+o_2-o_2'$. If $o_2\leq 1$ then $|\gamma_3(\GG)|\leq p$, and $G\cong H$ by \Cref{theorem1}\eqref{theorem1.1}. 
Thus we suppose  $o_2\geq 2$. Since $o_1'+o_2\leq m=4$, we derive that $o_1'=o_2=2$. Hence $a_1=o_1'=2$, and $o_2'\geq o_1'=2$, by condition \ref{3}. Since $o_2+o_2'\leq m =4$, necessarily $o_2'=2$. Hence we have that
$$ \inv(\GG)=(p,4,4,4,0,2,2,2,u_1^\GG,1)$$
with $1\leq u_1^\GG  \leq p^2$. 
Moreover, by \Cref{theorem1}\eqref{theorem1.2}, we have that $u_1^G \equiv u_1^H \mod p$. Hence there is an integer $1\leq i\leq p-1$ such that $u_1^\GG=i+j^\GG p$, for some integers  $ 0 \leq j^G, j^H\leq p-1$.  Finally, assume $i=p-1$, so $u_1^\GG \equiv -1 \mod p$. Since $q^\GG =6$, setting $t=1=2m-1-q^\GG$, \Cref{MIPUesHigherPower}\eqref{MIPUesHigherPower1} yields that $u_1^G=u_1^H$.
 \end{proof}

 \begin{remark}
 	 Observe that \Cref{theorem2} shows that $kG\cong kH$ implies $u_1^G\equiv u_1^H\mod p$ in almost all situations. A  pair of groups $G$ and $H$ of minimal size with $u_1^G\not \equiv u_1^H \mod p$ and not covered by this theorem (i.e., such that it is still open whether they have isomorphic group algebras or not) consists in groups of order $3^{17}$  and
 	\begin{align*}
 		\inv(G)=(3, 5, 7, 5, 1, 1, 2, 1, 1, 3, 1, 2); \\
 		\inv(H)=(3,5, 7, 5, 1, 1, 2, 1, 1, 3, 2, 2).
 	\end{align*}
 \end{remark}

\noindent\textbf{Acknowledgements}: We are grateful to Mima Stanojkovski, with whom we started the study of the Modular Isomorphism Problem for this class of groups, for useful comments and discussions on early drafts of this paper. \Cref{lemma:JTN}, if not folklore, was written by Sofia Brenner and the first author for another project: we are grateful to her for allowing us to include it here.

 \bibliographystyle{plain}
 \bibliography{MIP}

\begin{thebibliography}{10}

\bibitem{BaginskiMetacyclic}
C.~Bagi\'{n}ski.
\newblock The isomorphism question for modular group algebras of metacyclic
  {$p$}-groups.
\newblock {\em Proc. Amer. Math. Soc.}, 104(1):39--42, 1988.

\bibitem{BC88}
C.~Bagi\'{n}ski and A.~Caranti.
\newblock The modular group algebras of {$p$}-groups of maximal class.
\newblock {\em Canad. J. Math.}, 40(6):1422--1435, 1988.

\bibitem{BZ24}
C.~Bagi\'{n}ski and K~Zabielski.
\newblock The modular isomorphism problem -- the alternative perspective on
  counterexamples.
\newblock {\em \tt{arXiv:2406.01810}}, 2024.

\bibitem{Bra63}
R.~Brauer.
\newblock Representations of finite groups.
\newblock In {\em Lectures on {M}odern {M}athematics, {V}ol. {I}}, pages
  133--175. Wiley, New York, 1963.

\bibitem{BdR20}
O.~Broche and \'{A}. del R\'{\i}o.
\newblock The {M}odular {I}somorphism {P}roblem for two generated groups of
  class two.
\newblock {\em Indian J. Pure Appl. Math.}, 52:721--728, 2021.

\bibitem{OsnelDiegoAngel}
O.~Broche, D.~Garc\'{\i}a-Lucas, and \'{A}. del R\'{\i}o.
\newblock A classification of the finite 2-generator cyclic-by-abelian groups
  of prime-power order.
\newblock {\em International Journal of Algebra and Computation},
  33(04):641--686, 2023.

\bibitem{Dade71}
E.~Dade.
\newblock Deux groupes finis distincts ayant la m\^{e}me alg\`ebre de groupe
  sur tout corps.
\newblock {\em Math. Z.}, 119:345--348, 1971.

\bibitem{Deskins1956}
W.~E. Deskins.
\newblock Finite {A}belian groups with isomorphic group algebras.
\newblock {\em Duke Math. J.}, 23:35--40, 1956.

\bibitem{GLdRS2022}
D.~Garc\'{\i}a-Lucas, \'{A}. del R\'{\i}o, and M.~Stanojkowski.
\newblock On group invariants determined by modular group algebras: even versus
  odd characteristic.
\newblock {\em Algebr. Represent. Theory. \tt
  https://doi.org/10.1007/s10468-022-10182-x}, 2022.

\bibitem{GarciaMargolisdelRio}
D.~García-Lucas, L.~Margolis, and Á. del Río.
\newblock Non-isomorphic $2$-groups with isomorphic modular group algebras.
\newblock {\em J. Reine Angew. Math.}, 154(783):269--274, 2022.

\bibitem{Hertweck2001}
M.~Hertweck.
\newblock A counterexample to the isomorphism problem for integral group rings.
\newblock {\em Ann. of Math. (2)}, 154(1):115--138, 2001.

\bibitem{HigmanThesis}
G.~Higman.
\newblock {\em Units in group rings}.
\newblock 1940.
\newblock Thesis (Ph.D.)--Univ. Oxford.

\bibitem{Mar22}
L.~{Margolis}.
\newblock {The Modular Isomorphism Problem: A Survey}.
\newblock {\em Jahresber. Dtsch. Math. Ver.}, 2022.

\bibitem{MS24}
L.~Margolis and T.~Sakurai.
\newblock Identification of non-isomorphic 2-groups with dihedral central
  quotient and isomorphic modular group algebras.
\newblock {\em \tt{arXiv:2405.08075}}, 2024.

\bibitem{Passman1965p4}
D.~S. Passman.
\newblock The group algebras of groups of order {$p^{4}$} over a modular field.
\newblock {\em Michigan Math. J.}, 12:405--415, 1965.

\bibitem{Passman65}
D.~S. Passman.
\newblock Isomorphic groups and group rings.
\newblock {\em Pacific J. Math.}, 15:561--583, 1965.

\bibitem{Pas77}
D.~S. Passman.
\newblock {\em The algebraic structure of group rings}.
\newblock Pure and Applied Mathematics. Wiley-Interscience [John Wiley \&
  Sons], New York-London-Sydney, 1977.

\bibitem{Sandling85}
R.~Sandling.
\newblock The isomorphism problem for group rings: a survey.
\newblock In {\em Orders and their applications ({O}berwolfach, 1984)}, volume
  1142 of {\em Lecture Notes in Math.}, pages 256--288. Springer, Berlin, 1985.

\bibitem{San89}
R.~Sandling.
\newblock The modular group algebra of a central-elementary-by-abelian
  {$p$}-group.
\newblock {\em Arch. Math. (Basel)}, 52(1):22--27, 1989.

\bibitem{San96}
R.~Sandling.
\newblock The modular group algebra problem for metacyclic {$p$}-groups.
\newblock {\em Proc. Amer. Math. Soc.}, 124(5):1347--1350, 1996.

\bibitem{Usefi2008}
H.~Usefi.
\newblock Identifications in modular group algebras.
\newblock {\em J. Pure Appl. Algebra}, 212(10):2182--2189, 2008.

\bibitem{Ward}
H.~N. Ward.
\newblock Some results on the group algebra of a $p$-group over a prime field.
\newblock In {\em Seminar on finite groups and related topics.}, pages 13--19.
  Mimeographed notes, Harvard Univ., 1960-61.

\end{thebibliography}

\end{document}